\documentclass[final,onefignum,onetabnum]{siamart220329}

\usepackage{amssymb}
\usepackage{amsmath}
\usepackage{amsfonts}
\usepackage{algorithmic}
\usepackage{booktabs}
\usepackage{bm}
\usepackage{extpfeil}
\usepackage{epstopdf}
\usepackage{graphicx}
\usepackage{lipsum}
\usepackage{subfigure}
\ifpdf
  \DeclareGraphicsExtensions{.eps,.pdf,.png,.jpg}
\else
  \DeclareGraphicsExtensions{.eps}
\fi
\usepackage{amsopn}
\usepackage{bbm}
\usepackage{multirow}

\newsiamremark{remark}{Remark}
\newsiamremark{hypothesis}{Hypothesis}
\crefname{hypothesis}{Hypothesis}{Hypotheses}
\newsiamthm{claim}{Claim}

\newcommand{\dzeta}{\,\mathrm{d}\zeta}

\newcommand{\bbC}{\mathbb{C}}
\newcommand{\bbN}{\mathbb{N}}
\newcommand{\bbR}{\mathbb{R}}
\newcommand{\bbOne}{\mathbbm{1}}
\newcommand{\calD}{\mathcal{D}}
\newcommand{\calK}{\mathcal{K}}
\newcommand{\calR}{\mathcal{R}}

\newcommand{\calI}{\mathcal{I}}
\newcommand{\calO}{\mathcal{O}}
\newcommand{\calE}{\mathcal{E}}
\newcommand{\calG}{\mathcal{G}}
\newcommand{\calS}{\mathcal{S}}
\newcommand{\frakR}{\mathfrak{R}}

\headers{Interior Eigensolver}{Yuer Chen and Yingzhou Li}

\title{Interior Eigensolver Based on Rational Filter with Composite
rule}

\author{Yuer Chen\thanks{School of Mathematical Sciences, Fudan
University, Shanghai, China.} \and Yingzhou Li\thanks{School of
Mathematical Sciences, Fudan University, Shanghai,
China~(\email{yingzhouli@fudan.edu.cn}); Shanghai Key Laboratory
for Contemporary Applied Mathematics, Shanghai, China.}}

\begin{document}

\maketitle

\begin{abstract}
    Contour-integral-based rational filter leads to interior eigensolvers for non-Hermitian generalized eigenvalue problems. Based on Zolotarev's third problem, this paper proves the asymptotic optimality of the trapezoidal quadrature of the contour integral in terms of the spectrum separation. A composite rule of the trapezoidal quadrature is derived, and two interior eigensolvers are proposed based on it. Both eigensolvers adopt direct factorization and multi-shift generalized minimal residual method for the inner and outer rational functions, respectively. The first eigensolver fixes the order of the outer rational function and applies the subspace iterations to achieve convergence, whereas the second eigensolver doubles the order of the outer rational function every iteration to achieve convergence without subspace iteration. The efficiency and stability of proposed eigensolvers are demonstrated on synthetic and practical sparse matrix pencils.
\end{abstract}

\begin{keywords}
Generalized eigenvalue problem; non-Hermitian matrix; contour integral;
trapezoidal quadrature; rational optimality; Zolotarev's third problem.
\end{keywords}

\section{Introduction}

We aim to solve the large-scale interior eigenvalue problem for non-Hermitian matrices. Such problems arise from many fields including but not limited to electronic structure calculations, dynamic system simulations, control theory, etc. Most of these applications only require part of the eigenvalues of interest, and many of which are interior eigenvalues.
    
The interior non-Hermitian generalized eigenvalue problem we consider is
\begin{align}\label{eq: problem}
    A x_i = \lambda_i B x_i, \quad \lambda_i \in \calD,
\end{align}
where $\calD$ is the region of interest, matrix pencil $(A, B)$ is diagonalizable, and its eigenvalues are distributed on the complex plane. The goal is to find all eigenpairs $(\lambda_i, x_i)$ in the region $\calD$. The interior eigensolver for this problem can be used to compute eigenpairs in several regions in parallel to obtain the partial or full eigen-decomposition of the matrix of interest.

\subsection{Related work}

Methods for non-Hermitian generalized eigenvalue problems have been developed for decades. The QZ method~\cite{QZ} is a popular one in practice for dense and small-to-medium scale matrices. When a sparse and large-scale matrix is considered, iterative methods~\cite{Lehoucq1998ARPACKUG, krylovschur} are preferred. 

Among iterative methods, many adopt the combination of a contour-based filter and the subspace iteration, e.g., Sakurai-Sugiura~(SS) method~\cite{SAKURAI} 
and variants of FEAST method~\cite{li2021interior, FEASTorigin}. The original SS method suffers from numerical instability due to the ill-conditioned Hankel matrix. Then Sakurai and Sugiura propose CIRR~\cite{CIRR}, which uses Rayleigh-Ritz projection to avoid the explicit usage of the momentum and block version SS method~\cite{IKEGAMI20101927}. The number of linear systems therein is reduced, and so is the order of the Hankel matrix. The FEAST method
originally proposed for Hermitian matrices is extended to non-Hermitian matrices and results in many variants, dual FEAST~\cite{dualFEAST}, BFEAST~\cite{BFEAST}, HFEAST~\cite{HFEAST}, etc. 

For all the contour-based filters or rational filters in the methods above, the convergence and convergence rate highly depend on the locations and weights 
of poles. Although the trapezoidal quadrature\footnote{Trapezoidal rule instead of trapezoidal quadrature is a commonly used terminology.} leads to a good convergence behavior~\cite{dualFEAST}, its optimality remains unknown for non-Hermitian 
matrices. In this paper, we discuss the optimality of the trapezoidal quadrature and its composite rule property. Based on the composite rule, we propose two interior eigensolvers for non-Hermitian generalized eigenvalue problems.

\subsection{Contribution}

Our contribution in this paper has two parts, theory part and algorithm part. These two parts are later referred as the filter design and filter implementation, respectively.

In filter design, we find the optimal rational filter in the sense of spectrum separation for fast convergence of subspace iteration. By making use of the connection with Zolotarev's third problem, we prove that when the contour is a circle, the rational filter used in the inverse power method leads to an optimal separation, while the trapezoidal quadrature leads to an asymptotically optimal separation. 

In filter implementation, we focus on the flexibility of the
trapezoidal quadrature implementation. The main cost of the rational
filter $R_k(z)$ implementation comes from pre-factorization, e.g., LU
factorization, and solving phase, e.g., forward and backward
substitution (triangular solve). In general, the cost of the
pre-factorization is of higher order complexity than that of the
solving phase. The composite rule of trapezoidal quadrature
establishes a trade-off between the cost of two.

Specifically, given a rational filter $R_k(z)$ from the trapezoidal quadrature, we derive a composite rule as $R_k(z) = R_{k_2}(M(R_{k_1}(z)))$ for $k = k_1 k_2$ and $M(\cdot)$ being a M{\"o}bius transform. Motivated by~\cite{li2021interior}, two novel algorithms are proposed based on the composite rule, both of which implement $R_{k}(\cdot)$ with an inner-outer structure. The inner rational function $R_{k_1}(\cdot)$ is implemented with direct matrix factorizations, whereas the outer rational function $R_{k_2}(\cdot)$ is implemented via the multi-shift generalized minimal residual method~(GMRES). The first algorithm adopts the subspace iteration framework. It substitutes matrix factorizations with the solves, which may reduce the total runtime for large cases. The second algorithm discards the framework of subspace iteration. It achieves target precision by dynamically increasing $k_2$ and reusing Krylov subspace to avoid the increase of the computational cost. These two algorithms hybridize the direct method and iterative method and enable us to use computational resources more effectively. Through the numerical experiments, we find that the second algorithm is more efficient and practical. 

\subsection{Organization}

The rest of this paper is organized as follows. In \cref{sec: Preliminary}, we introduce the basic idea of the contour-integral-based filter and discuss the computational cost and memory cost of the implementation based on the simple rule of trapezoidal quadrature. Later, we introduce our two novel algorithms based on the composite rule of trapezoidal quadrature in \cref{sec: eigensolver}. The optimality of rational functions is discussed in \cref{sec: Optimal} and the derivation of the composite rule is presented in \cref{sec: composite}. In \cref{sec: experiment}, there are some numerical experiments to examine our results of optimality and demonstrate the efficiency of two proposed algorithms. Finally, \cref{sec: conclusion} concludes the paper.

\section{Subspace iteration with rational filter}
\label{sec: Preliminary}

Subspace iteration with rational filter is a class of eigensolvers for interior non-Hermitian generalized eigenvalue problems~\cref{eq: problem}. All eigensolvers in this class use the subspace iteration framework and adopt various filters, i.e., rational functions with different choices of weights and poles. In this section, we will first review the framework of subspace iteration and introduce the contour integral formulation of the spectral projector. Then we demonstrate the process of deriving rational filters from numerical discretizations, alongside the cost of implementing this kind of rational filter.

\subsection{Subspace iteration}\label{subsec:SIintro}

The general framework of the subspace iteration with filter iterates between two phases: 1) refining the subspace via filter, and 2) solving a reduced eigenvalue problem in the subspace.

In the first phase, the filter is applied to an approximate basis of the eigensubspace, and a refined representation of the eigen-subspace is obtained. For non-Hermitian eigenvalue problems, left and right eigen-subspaces are different. We can refine left and right eigen-subspaces by applying the filter twice~\cite{dualFEAST}, or we can only refine the right eigen-subspace and use 
an extra step to obtain an approximation of the left eigen-subspace~\cite{HFEAST}. In the second phase, the original large-scale eigenvalue problem is projected to the left and right eigen-subspaces and reduced to an eigenvalue problem of a much smaller scale. Then the small-scale eigenvalue problem is solved by classical dense eigensolvers, which results in the approximate eigenpairs of the original problem.

Due to the potentially ill-conditioned eigenbasis of non-Hermitian matrices, the generalized Schur vectors could be extracted to represent the eigen-subspaces and lead to a more stable scheme. Such a subspace iteration idea has been combined with FEAST for non-Hermitian matrices and results in HFEAST~\cite{HFEAST}. Denote the approximate basis of the right eigen-subspace as $U$. The orthonormal basis of $U$ is denoted as $V = \mathrm{orth}(U)$. As in HFEAST~\cite{HFEAST}, the orthonormal basis of the left eigen-subspace could be constructed as $W = \mathrm{orth}(AV - \sigma BV)$, where $\sigma$ is the shift away from the eigenvalues of $(A, B)$. After obtaining the approximate orthonormal basis of the left and right eigen-subspaces, the reduced generalized eigenvalue problem $(W^* A V, W^*B V)$ is solved by the QZ algorithm and yields the generalized Schur form,
\begin{equation*}
    P_L^* (W^*AV) P_R = H_A \quad \text{and} \quad P_L^* (W^*BV) P_R = H_B,
\end{equation*}
where $P_L$ and $P_R$ are orthogonal matrices, $H_A$ and $H_B$ are upper
triangular matrices. The approximate eigenvalues are,
\begin{equation*}
    \tilde{\lambda}_i = (H_A)_{i,i} / (H_B)_{i,i}.
\end{equation*} 
We further calculate the left and right eigenvectors 
of $(H_A, H_B)$ and denote them as $V_L$ and $V_R$ respectively. 
The approximate left and right eigenvectors of $(A, B)$ are, respectively,
\begin{equation*}
    W P_L V_L \quad \text{and} \quad V P_R V_R.
\end{equation*}

\begin{algorithm}
    \caption{Subspace Iteration with Filter}
    \label{alg: SIsketch}
    \begin{algorithmic}[1]

    \REQUIRE matrix pencil $(A,B)$, region $\calD$, number of columns
    $n_\mathrm{col}$, shift $\sigma$.

    \ENSURE All approximate eigenpairs 
    $(\widetilde{\lambda}_i,\widetilde{x}_i),\tilde{\lambda}_i \in \calD$.
    \STATE Generate random $Y^{N\times n_{\mathrm{col}}}$
    \WHILE {not converge}
        \STATE $U = \rho(B^{-1}A)Y$
        \STATE $V = \mathrm{orth}(U)$
        \STATE $W = \mathrm{orth}(AV - \sigma BV)$
        \STATE $[H_A, H_B, P_L, P_R, V_L, V_R] = \mathrm{qz}(W^*AV, W^*BV)$
        \STATE $\widetilde{\lambda}_i = (H_A)_{i,i}/(H_B)_{i,i}$, 
        $\widetilde{x}_i=V P_R (V_R)(:,i)$ and 
        $Y = (\widetilde{x}_1,\dots,\widetilde{x}_{n_{\mathrm{col}}})$
    \ENDWHILE
    \end{algorithmic}
\end{algorithm}

The overall framework of the subspace iteration in HFEAST~\cite{HFEAST} with filter $\rho(\cdot)$ is summarized in \cref{alg: SIsketch}. In the rest of the paper, we adopt the subspace iteration as in \cref{alg: SIsketch} and focus on the design and implementation of $\rho(\cdot)$. Besides, we abbreviate the lines 4 to 7 as $[Y,\widetilde{\Lambda},\widetilde{X}]=\mathrm{HSRR}(A, B, U, \sigma)$, where $\widetilde{\Lambda}=\mathrm{diag}\{\widetilde{\lambda}_1,\dots,\widetilde{\lambda}_{n_\mathrm{col}}\}$ and
$\widetilde{X}=(\widetilde{x}_1,\dots,\widetilde{x}_{n_\mathrm{col}})$.

\subsection{Contour-based filter and discretization}
\label{subsec:discretization}

The basic idea of designing a filter is to construct a matrix function whose eigenvalues are close to zero outside the region $\calD$ and different from zero inside $\calD$. One good choice of matrix functions is the indicator function of $\calD$, which could be constructed via a contour integral enclosing the region $\calD$. The indicator function of $\calD$ via contour integral admits,
\begin{equation} \label{eq: CI}
    f(z) = \frac{1}{2 \pi \imath} \oint_{\Gamma} \frac{1}{\zeta-z}
    \dzeta = 
    \begin{cases}
        1, & z \in \calD\\
        0, & z \not\in \bar{\calD}
    \end{cases},
\end{equation}
where $\Gamma$ is the positively oriented Jordan curve encloses the region $\calD$.

For a diagonalizable matrix pencil $(A, B)$, i.e.,
\begin{equation*}
    A X = B X \Lambda,
\end{equation*}
with $X$ being the right eigenvectors and $\Lambda$ is a diagonal matrix with eigenvalues on its diagonal, the indicator function $f(z)$ applying to $(A,B)$ admits
\begin{equation} \label{eq: CIoperator}
    \begin{split}
        f(B^{-1}A) & = \frac{1}{2\pi\imath} \oint_{\Gamma}
        (\zeta B - A)^{-1} B \dzeta \\
        & = X \bigg[ \frac{1}{2\pi\imath}
        \oint_{\Gamma}(\zeta I - \Lambda)^{-1} \dzeta \bigg] X^{-1}
        = X \bbOne_\calD (\Lambda) X^{-1},
    \end{split}
\end{equation}
where $\bbOne_\calD(\cdot)$ denotes the indicator function for region $\calD$~\footnote{In \eqref{eq: CIoperator}, we implicitly assume that the eigenvalues of $(A, B)$ do not locate on the boundary of $\calD$.}. In~\cite{BFEAST}, a result similar to \cref{eq: CIoperator} is proved, which leads to the theoretical foundation that the contour integral works even if the non-Hermitian system is defective. Various numerical discretizations of the contour integral \cref{eq: CIoperator} lead to various filters. In many applications, especially non-Hermitian eigenvalue problems, the contour $\Gamma$ is circular. In other applications, the contour could be conformally mapped to a circle. Hence, in this paper, we focus on the case that $\Gamma$ is a circle. 

We could reparameterize the circle by $e^{\imath \theta}$ for $\theta \in [0, 2\pi)$. The integral \cref{eq: CI}, then, is a one-dimensional integral and could be numerically evaluated by various quadrature rules. Generally, the contour integral \cref{eq: CI} can be approximated by a numerical quadrature with $k$ quadrature points, written as 
\begin{equation} \label{eq: discretizedCI}
    R_{k}(z) = \sum_{i=1}^k \frac{w_i^{(k)}}{p_i^{(k)} - z},
\end{equation}
where $\{w_i^{(k)}\}_{i=1}^k$ are weights, $\{p_i^{(k)}\}_{i=1}^k$ are poles. Let 
\begin{equation*}
    \calR_{n,m} = \{P(z)/Q(z):\mathrm{deg}(P(z))\le n, \mathrm{deg}(Q(z))\le m\}
\end{equation*}
be the set of rational functions, where $P(z)$ and $Q(z)$ are polynomials and $\mathrm{deg}(\cdot)$ denotes the degree of the polynomial. It is easy to see that $R_{k}(z)\in\calR_{k,k}$, i.e., it is a $k$-th order rational function.

When the discretized contour integral is applied to $(A,B)$, it yields a $k$-th order rational matrix function 
\begin{equation} \label{eq: discretizedCIoperator}
    R_{k}(B^{-1}A) = \sum_{i=1}^k w_i^{(k)} (p_i^{(k)} B - A)^{-1}B.
\end{equation}
One of the common choices is the trapezoidal quadrature. When the contour is a circle, whose center is $c$ and radius is $r$, the integral \cref{eq: CI}  numerically discretized with $k$ quadrature points is denoted as $R_{c,r,k}(B^{-1}A)$, whose poles and weights are
\begin{equation}\label{eq: pwtp}
    p_i^{(k)} = re^{\imath\theta_i^{(k)}}+c,\quad w_i^{(k)} = re^{\imath\theta_i^{(k)}}/k,\quad \theta_i^{(k)} = (2i - 1) \pi/k.
\end{equation}
The matrix function $R_{c,r,k}(B^{-1}A)$ is used as the filter in this paper. We will call it the $k$-th order trapezoidal quadrature.

\subsection{Cost of implementation}\label{subsec:practicalconsideration}

In \cref{eq: discretizedCIoperator}, the considerable computational burden lies in solving the shifted linear systems, $(p_i^{(k)} B - A)^{-1}$ for $i = 1, \dots, k$. When the eigengap between the interior eigenvalues and outer eigenvalues is small, the linear systems will be ill-conditioned,
since the poles are on the contour and their distances to eigenvalues are
bounded by the eigengap. Hence, in most contour-based filters, the shifted linear systems are solved by direct methods. 

The overall computational cost is then divided into two parts: the offline factorization part, e.g., LU factorization, and the online solving part, e.g., triangular solve with the given LU factorization. The computational cost could be written as
\begin{equation*}
    C_{\mathrm{factor}} \times k + C_{\mathrm{apply}} \times k \times
    n_{\mathrm{col}} \times T + o(C_{\mathrm{apply}}),
\end{equation*} 
where $C_{\mathrm{factor}}$ is the cost of a factorization, $C_{\mathrm{apply}}$ is the cost of a solve, $k$ is the number of poles, $n_{\mathrm{col}}$ is the number of columns in $Y$, $T$ is the number of subspace iterations, and $o(C_{\mathrm{apply}}) = o(C_{\mathrm{apply}}(N))$ is the rest cost lower order. To extract the entire eigen-subspace we are interested, it is necessary that $n_\mathrm{col}\ge s$ for $s$ being the number of eigenvalues inside. As for the memory cost, the solving complexity is the same as its memory cost for almost all dense and sparse linear system solvers. Hence, $C_{\mathrm{apply}}$ is also used as the memory cost in storing a factorization. The cost of memory is $kC_{\mathrm{apply}}$.

Throughout the subspace iterations, the tuneable hyperparameters are $k$ and $n_{\mathrm{col}}$. The dependence of $T$ on $k$ and $n_{\mathrm{col}}$ could be reflected by the function value $R_{k}(\lambda_i)$ since we are essentially applying power method with $R_{k}(B^{-1}A)$. Let $\sigma$ be a permutation of $1, 2, \dots, N$, such that
\begin{equation*}
    |R_{k}(\lambda_{\sigma_1})| \geq |R_{k}(\lambda_{\sigma_2})| \geq
    \cdots \geq |R_{k}(\lambda_{\sigma_N})|.
\end{equation*}
Then, the number of subspace iterations $T$ mainly depends on the ratio,
\begin{equation} \label{eq: ratioeig}
    \max_{i > n_{\mathrm{col}}}
    |R_{k}(\lambda_{\sigma_i})| \Bigm/ \min_{\lambda_{\sigma_i} \in \calD}
    |R_{k}(\lambda_{\sigma_i})|.
\end{equation}
When the ratio is greater or equal to one, the subspace iteration would suffer from a divergence issue. When the ratio is less than one, the smaller the ratio, the faster the convergence. In general, larger $k$ and $n_{\mathrm{col}}$ leads to smaller \cref{eq: ratioeig}. However, $k$ is limited by the bottleneck of memory, especially for the large-scale problems. Then the implementation based on the simple rule can only increase $n_{\mathrm{col}}$ to make the ratio \cref{eq: ratioeig} smaller than one. When many outer eigenvalues are very close to the contour and $k$ is small, an extremely large $n_{\mathrm{col}}$ is needed for the convergence. 

We call the formula \cref{eq: discretizedCIoperator} of $R_{c,r,k}(B^{-1}A)$ as the simple rule, which differs from the composite rule introduced in the next section.

\section{Two eigensolvers based on the composite rule}\label{sec: eigensolver}
The implementation based on the simple rule \cref{eq: discretizedCIoperator} needs $k$ matrix factorizations, which brings huge burden on computation and memory for large matrices. In this section, we will start by introducing the composite rule of trapezoidal quadrature without derivation. Then we propose two eigensolvers based on the composite rule, both of which implement $R_{c,r,k}(B^{-1}A)$ with an inner-outer structure.
\subsection{The composite rule}
Given a positive integer $k$ and its integer factorization $k = k_1 k_2$ for $k_1 > 1$ and $k_2 > 1$, the composite rule of trapezoidal quadrature is shown as 
\begin{equation*}
    R_{c,r,k}(z) = R_{0,1,k_2}(M ( R_{c,r,k_1}(z)) ),
\end{equation*}
where $M(z) = (1-z)/z$ is a M\"obius transform. 
That means the $k$-th order trapezoidal quadrature can be rewritten as a composition of a $k_2$-th order trapezoidal quadrature and a tranformed $k_1$-th order trapezoidal quadrature. 

When $k_2$ is even, the composite rule can be rewritten as 
\begin{equation}\label{eq: Rcrm_}
    \begin{aligned}
        &R_{c,r,k}(z) =\sum_{i=1}^{k_2} c_i^{(k_2)}(R_{c,r,k_1}(z) - s_i^{(k_2)})^{-1}R_{c,r,k_1}(z),\\
        &c_i^{(k_2)} = -\frac{1}{k_2}\frac{\sigma_i^{(k_2)}}{1+\sigma_i^{(k_2)}},~s_i^{(k_2)} = \frac{1}{1+\sigma^{(k_2)}_i},
    \end{aligned}
\end{equation}
where $\{\sigma_i^{(k_2)}\}_{i=1}^{k_2}$ are roots of $z^{k_2}=-1$.
\subsection{Interior eigensolver with subspace iteration}
\label{subsec: eigensolver}

Using $R_{c, r, k}(z)$ as the filter in subspace iteration for
a matrix pencil $(A, B)$ requires the evaluation of $R_{c, r, k}(B^{-1} A)
Y$ for $Y$ being a matrix of size $N \times n_{\mathrm{col}}$. By the
composite rule for $R_{c, r, k}(z)$ in \cref{eq: Rcrm_}, the evaluation of
$R_{c, r, k}(B^{-1}A)Y$ could be rewritten as,
\begin{equation} \label{eq: linearsolve1}
    R_{c, r, k}(B^{-1}A)Y = \left( \sum_{i=1}^{k_2}
    c_i^{(k_2)}(R_{c,r,k_1}(B^{-1} A) - s_i^{(k_2)} I)^{-1} \right)
    \left( R_{c,r,k_1}(B^{-1} A) Y \right).
\end{equation}
where the operation $R_{c,r,k_1}(B^{-1} A) Y$ can be written as,
\begin{equation} \label{eq: linearsolve2}
    R_{c, r, k_1}(B^{-1}A)Y = \sum_{i=1}^{k_1}w_i^{(k_1)}(p_i^{(k_1)}B-A)^{-1}BY
\end{equation}
for $\{w_i^{(k_1)}\}$ and $\{p_i^{(k_1)}\}$ being the weights and poles of $R_{c, r,
k_1}(\cdot)$.

In \cref{eq: linearsolve1}, there are inner and outer parts of rational function evaluations. For the inner part, as in \cref{eq: linearsolve2}, we use direct solvers for all these linear systems for the same reason as the simple rule. We pre-factorize all linear systems and denote them as $K_i = p_i^{(k_1)} B - A$ for $i = 1, \dots, k_1$, e.g., $K_i=L_iU_i$.~\footnote{Throughout the numerical section of this paper, dense LU factorization is used by default for dense matrices $A$ and $B$. If $A$ and $B$ are sparse matrices, we adopt the default sparse LU factorization methods in MATLAB.} Once the factorizations are available, the inner part could be addressed efficiently. The inner part \cref{eq: linearsolve2} essentially applies a rational filter of the matrix pencil $(A, B)$ to a set of vectors $Y$. Without loss of generality, we treat the inner part as an operator $G$ acting on $Y$.

For the outer part, we first rewrite \cref{eq: linearsolve1} using the operator $G$,
\begin{equation} \label{eq: linearsolve1abs}
    R_{c, r, k}(B^{-1}A) Y = \sum_{i=1}^{k_2}c_i^{(k_2)}(G - s_i^{(k_2)} I)^{-1}\widetilde{Y}
\end{equation}
for $\widetilde{Y} = G(Y)$. We notice that the eigenvalues of $G$ are clustered, i.e., the eigenvalues outside $\calD$ cluster around zero and the eigenvalues inside $\calD$ cluster around one, see \cref{fig: mobius}. Iterative solvers, especially GMRES, are expected to converge fast. Throughout this paper, we adopt GMRES~\cite{GMRES} as the default iterative solver for \cref{eq: linearsolve1abs} with $G$ being applied as an operator. Recall that GMRES is a Krylov subspace method, by the shift-invariant property of the Krylov subspace, all $k_2$ shifts in \cref{eq: linearsolve1abs} could be addressed simultaneously in the same Krylov subspace, i.e.,
\begin{equation*}
    \begin{split}
        & \calK_n(G-s_i^{(k_2)}I,y)=\calK_n(G,y), \\
        & (G-s_i^{(k_2)}I)V_n=V_n(H_{n,n+1}-s_i^{(k_2)}I_{n,n+1}),
    \end{split}
\end{equation*}
for $i = 1, \dots, k_2$ and $V_n$ denoting the basis of $\calK_n(G, y)$. The multi-shift GMRES~\cite{multihiftGMRES} applies the operator $G$ once per iteration. In all of our numerical experiments, the multi-shift GMRES converges in less than one hundred iterations, and no restarting is needed. 

Using a direct solver and an iterative solver for the inner and outer part
of \cref{eq: linearsolve1}, we obtain an effective algorithm for the
rational matrix function filter. Combining this filter with subspace
iteration leads to our first eigensolver. \Cref{alg: composite} gives the
overall pseudocode, where HSRR is an abbreviation for lines 4 to 7 of
\cref{alg: SIsketch}.

The convergence criterion for interior eigenvalue problem can be quite tricky. The relative error of eigenpair in our algorithms is defined as 
\begin{equation}\label{eq: rel-error}
    e_i=e(\tilde{\lambda}_i, \tilde{x}_i) = \frac{\Vert A \tilde{x}_i -
    B \tilde{x}_i \tilde{\lambda}_i \Vert_2}{(|c| + r)\Vert B \tilde{x}
    \Vert_2},
\end{equation}
where $c$ and $r$ is the center and radius of the region $\calD$. For the non-Hermitian interior eigenvalue problem, a phenomenon called ghost eigenvalue often appears. The ghost eigenvalue is
the one that appears in the region $\mathcal{D}$ but does not converge. The ghost eigenvalue would make it difficult to examine the convergence of subspace iterations.

One of the practical strategies is to set a tolerance  $\tau_g$ as in~\cite{BFEAST}, which is much larger than the target relative error $\tau$. As the iteration goes, the true eigenvalues will converge to a small relative error, while the ghost eigenvalues will not converge to the same precision. After a few steps, there is a gap in the relative errors between the true eigenvalues and the ghost eigenvalues. When the relative error of an approximate eigenpair $(\tilde{\lambda}_i,\tilde{x}_i)$ inside $\mathcal{D}$ is smaller than $\tau_g$, we treat it as a filtered eigenpair and denote the number of filtered eigenpairs as $p$. When $p$ is not changed and all relative errors of the filtered eigenpairs are smaller than $\tau$, we terminate the algorithm. That corresponds to the 7th and 12th lines of \cref{alg: composite}.

\begin{algorithm}[htbp]
    \caption{Eigensolver: Composite rational function filter}
    \label{alg: composite}
    \begin{algorithmic}[1]
        \REQUIRE Pencil $(A,B)$, center $c$, radius $r$, number of
        eigenvalues $s$, shift $\sigma$, number of poles $[k_1, k_2]$, tolerance $\tau_g$ and $\tau$.

        \ENSURE The approximate eigenpair $(\tilde{\lambda}_i,\tilde{x}_i)$ 
        with $\tilde{\lambda}_i \in \calD$.

        \STATE Compute $\{p_i^{(k_1)}, w_i^{(k_1)}\}_{i=1}^{k_1}$, $\{c_j^{(k_2)},
        s_j^{(k_2)}\}_{j=1}^{k_2}$.

        \FOR {$i = 1, \cdots, k_1$}
            \STATE Pre-factorize $p_i^{(k_1)} B - A$ as $K_i$. (e.g., $K_i=L_iU_i$.)
        \ENDFOR

        \STATE Construct a function for applying $G$ to a set of vectors $V$.
        \begin{equation*}
            G(V) = \sum_{i=1}^{k_1} w_i^{(k_1)} K_i^{-1} B V.~(\mathrm{e.g.},~K_i^{-1}=U_i^{-1}L_i^{-1}.)
        \end{equation*}

        \STATE Generate an orthonormal random matrix $Y^{N \times
        n_{\mathrm{col}}}$ with $n_{\mathrm{col}}\ge s$.

        \WHILE {$p$ changes or any $e_i$ of filtered eigenpair is larger than $\tau$}
            \STATE $\widetilde{Y} = G(Y)$.

            \STATE Solve $U_j = (G-s_j^{(k_2)}I)^{-1}\widetilde{Y}$ for $j = 1,
            \cdots, k_2$ via multi-shift GMRES.

            \STATE $U = \sum_{j=1}^{k_2}c_j^{(k_2)} U_j$.

            \STATE $[Y,\widetilde{\Lambda},\widetilde{X}]=
            \mathrm{HSRR}(A,B,U,\sigma)$.

            \STATE Distinguish ghost eigenvalues and filtered eiegnvalues by $\tau_g$. Update the number of filtered eigenvalues $p$. 
        \ENDWHILE
    \end{algorithmic}
\end{algorithm}

We now estimate the computational cost for \cref{alg: composite}. In the preparation phase before subspace iteration, the weights and poles are computed in the computational cost of $O(1)$. For the pre-factorizations of $k_1$ linear systems, the computational complexity is $k_1 C_{\mathrm{factor}}$ and the memory required is $k_1 C_{\mathrm{apply}}$. In the subspace iteration phase, the per-iteration computational cost is dominated by the multi-shift GMRES. If we denote $n_{\mathrm{iter}}^{(j,t)}$ as the GMRES iteration number for $j$-th
column in the $t$-th subspace iteration, the dominant computational cost in the GMRES is
\begin{equation*}
    \sum_{t = 1}^T \sum_{j = 1}^{n_{\mathrm{col}}}
    n_{\mathrm{iter}}^{(j,t)} \cdot k_1 C_{\mathrm{apply}},
\end{equation*}
where $k_1 C_{\mathrm{apply}}$ is the cost in applying $G(\cdot)$ to a vector. 

The overall dominant computational and memory costs for \cref{alg: composite} are summarized in \cref{tab: compare}. In the same table, we also list the computational and memory costs for subspace iteration with $k_1k_2$-th order rational filter without using the composite rule. Another row of ratio is added to indicate the acceleration from \cref{alg: composite}. Clearly, both the computation and memory costs in the pre-factorization phase are reduced by a factor of $k_2$. While the comparison for the subspace iteration part is less clear. The ratio depends on the iteration numbers of both the subspace iteration and the multi-shift GMRES. Another interesting
thing is that as the subspace iteration goes, the columns of $Y$
become closer aligned with the eigenvectors, which means the Krylov
subspaces will converge faster and so will the GMRES, as we
numerically shown in \cref{sec: decayGMRES}.

\begin{table}[htb]
    \centering
    \caption{Computational and memory complexities of the subspace iteration with the simple rational filter and the composite rational filter. The simple rational filter is of order $k_1 k_2$ and the composite rational filter is of inner and outer order $k_1$ and $k_2$ respectively. Here $C_\mathrm{factor}$ and $C_\mathrm{apply}$ are factorization and solving cost for a matrix of size $N \times N$.}
    \label{tab: compare}
    \begin{tabular}{ccccc}
        \toprule
        \multirow{2}{*}{Algorithm} & \multicolumn{2}{c}{Computation}
        & \multicolumn{2}{c}{Memory} \\
        & Pre-Fact & Iteration & Pre-Fact & Iteration \\
        \toprule
        Simple               &
        $k_1k_2C_{\mathrm{factor}}$ &
        $T n_{\mathrm{col}} k_1 k_2 C_{\mathrm{apply}}$ &
        $k_1k_2C_{\mathrm{apply}}$ & $n_\mathrm{col} N$ \\
        \Cref{alg: composite} &
        $k_1C_{\mathrm{factor}}$ & $\sum\limits_{t=1}^T
        \sum\limits_{j=1}^{n_{\mathrm{col}}}
        n_{\mathrm{iter}}^{(j,t)} k_1C_{\mathrm{apply}}$ &
        $k_1C_{\mathrm{apply}}$ &
        $\max\limits_{t=1}^T \sum\limits_{j=1}^{n_\mathrm{col}}
        n_\mathrm{iter}^{(j,t)} N$\\
        \midrule
        Ratio & $k_2$ & $\frac{T n_{\mathrm{col}} k_2}{\sum\limits_{t=1}^T
        \sum\limits_{j=1}^{n_{\mathrm{col}}} n_{\mathrm{iter}}^{(j,t)}}$ &
        $k_2$ & $\frac{n_\mathrm{col}}{
        \max\limits_{t=1}^T \sum\limits_{j=1}^{n_\mathrm{col}}
        n_\mathrm{iter}^{(j,t)}}$\\
        \bottomrule
    \end{tabular}
\end{table}

\subsection{Composite rule eigensolver without subspace iteration}
\label{subsec: compositewithoutsubspaceiteration}

The \cref{alg: composite} still adopts
the framework of subspace iteration. It implements the same order
trapezoidal quadrature as the simple rule in a different way. We can
take advantage of the composite rule from another point of view, i.e.,
achieving better approximation with the same number of factorizations
as the simple rule. The better the rational approximation, the fewer
subspace iteration is needed. In the limit of very accurate
approximation, only one subspace iteration is enough. Making use of
the shift-invariant property of Krylov subspace, we proposed \cref{alg: composite2}, which achieves better approximation with a fixed number of
factorizations and discards the framework of subspace iteration.

More specifically, the first step of \cref{alg: composite2} for the initial $[k_1,k_2]$ is the same as \cref{alg: composite}, which also constructs the operator $G$ via pre-factorizations and generates Krylov subspaces of different vectors for the computation of $R_{c,r,k}(Y)$. When the eigenpairs do not converge, \cref{alg: composite2} will double $k_2$ and compute new shifts and weights, $s_j^{(2k_2)}$ and $c_j^{(2k_2)}$. Importantly, \cref{alg: composite2} does not regenerate new Krylov subspaces from the approximate eigenvectors. Instead, it computes $R_{c,r,2k}(Y)$ in the existing Krylov subspaces used for $R_{c,r,k}(Y)$ and expands them when it is necessary. The \cref{alg: composite2} keeps enlarging $k_2$ until all the eigenpairs converge. As we will show in \cref{sec: experiment}, the dimension of Krylov subspace is not sensitive to $k_2$ and increases mildly. In addition, we find that the shift $s_j^{(k_2)}$ are parts of the shifts $s_j^{(2k_2)}$ and their weights satisfy $c_j^{(k_2)}/2=c_j^{(2k_2)}$. This means we only need to compute $U_j=(G-s_j^{(2k_2)}I)^{-1}\widetilde{Y}$ for the new shifts, then $R_{c,r,2k}(Y)$ can be computed from $R_{c,r,2k}(Y) = \frac{1}{2} R_{c,r,k}(Y) + \sum_{j=k_2+1}^{2k_2} c_j^{(2k_2)} U_j$. It corresponds to the 9th and 10th lines of \cref{alg: composite2}.

\begin{algorithm}[htb]
    \caption{Eigensolver: Composite rational function filter without
    subspace iteration}
    \label{alg: composite2}
    \begin{algorithmic}[1]
        \REQUIRE Pencil $(A,B)$, center $c$, radius $r$, number of
        eigenvalues $s$, shift $\sigma$, tolerance $\tau_g$ and $\tau$, number of poles $k_1$, initial $k_2$ ($k_2=k_1$ in default), and $\hat{k}_2 = 0$.

        \ENSURE The approximate eigenpair $(\tilde{\lambda}_i,\tilde{x}_i)$ 
        with $\tilde{\lambda}_i \in \calD$.

        \STATE Compute $\{p_i^{(k_1)}, w_i^{(k_1)}\}_{i=1}^{k_1}$, $\{c_j^{(k_2)},
        s_j^{(k_2)}\}_{j=1}^{k_2}$.

        \FOR {$i = 1, \cdots, k_1$}
            \STATE Pre-factorize $p_i^{(k_1)} B - A$ as $K_i$. (e.g., $K_i=L_iU_i$.)
        \ENDFOR

        \STATE Construct a function for applying $G$ to a set of vectors $V$.
        \begin{equation*}
            G(V) = \sum_{i=1}^{k_1} w_i^{(k_1)} K_i^{-1} B V.~(\mathrm{e.g.},~K_i^{-1}=U_i^{-1}L_i^{-1}.)
        \end{equation*}

        \STATE Generate an orthonormal random matrix $Y^{N \times
        n_{\mathrm{col}}}$ with $n_{\mathrm{col}}\ge s$.

        \STATE $\widetilde{Y} = G(Y)$, $U$ be a zero matrix of the same
        size.

        \WHILE {$p$ changes or any $e_i$ of filtered eigenpair is larger than $\tau$}
            \STATE Solve $U_j = (G-s_j^{(k_2)}I)^{-1}\widetilde{Y}$ for $j=\hat{k}_2+1,\cdots,k_2$ via multi-shift GMRES in the existing Krylov subspaces and expand it when necessary.

            \STATE $U = U/2 + \sum_{j=\hat{k}_2+1}^{k_2}c_j^{(k_2)}U_j$.

            \STATE $[Y,\widetilde{\Lambda},\widetilde{X}]=
            \mathrm{HSRR}(A,B,U,\sigma)$. 

            \STATE Distinguish ghost eigenvalues and filtered eiegnvalues by $\tau_g$. Update the number of filtered eigenvalues $p$.

            \STATE $\hat{k}_2=k_2, k_2=2k_2$. Compute $\{c_j^{(k_2)},s_j^{(k_2)}\}_{j=\hat{k}_2+1}^{k_2}$.
        \ENDWHILE
    \end{algorithmic}
\end{algorithm}

Compared to \cref{alg: composite}, \cref{alg: composite2} does not regenerate Krylov subspaces each time and enables adaptive selection of $k_2$, which makes \cref{alg: composite2} more practical. The more interesting characteristic of \cref{alg: composite2} is that it discards the framework of subspace iteration. We find that the idea of reusing Krylov subspace for algorithm design is also shown in~\cite{RIM}, where they use a single Cayley transform for preconditioning. Instead, we use trapezoidal quadrature with $k_1$ poles for preconditioning. That means \cref{alg: composite2} can enjoy the benefit of parallelization, as the solving of the $k_1$ shifted linear systems $(p_i^{(k_1)}B-A)x=y$ can be performed simultaneously.

Another feature of \cref{alg: composite2} is that in the whole process of dynamically increasing $k_2$, we are essentially using the trapezoidal quadrature as a filter. While for \cref{alg: composite} and the simple rule, both of them use powers of the trapezoidal quadrature. As we show in the next section, trapezoidal quadrature is asymptotically optimal, which contributes part of the advantages of \cref{alg: composite2}.
\section{Asymptotically optimal contour discretization}
\label{sec: Optimal}

In this section, we will start with the definition of spectrum separation, which is a continuous version of \cref{eq: ratioeig}. Then we will study the optimal and asymptotically optimal rational filter in the sense of spectrum separation. Two main results will be proved, i.e., the rational function used in the inverse power method is optimal and the trapezoidal quadrature is asymptotically optimal with respect to $k$. The latter one shows that $R_{c,r,k}(z)$ is a reasonable good choice.

\subsection{Spectrum separation}
Among various quadrature rules, the optimality of quadrature needs to be defined based on a criterion. The convergence rate mainly depends on \cref{eq: ratioeig}. Since we do not know eigenvalues in a priori, we could assume that there is an annulus around the boundary of $\calD$ as a generalized eigengap of the inner and outer eigenvalues. The inner part and the outer part are,
\begin{equation*}
    \calI = \{z:|z|\le a\}, \quad \text{and } \calO=\{z:|z|\ge b\},
\end{equation*}
where $a$ and $b$ are the radii of the inner and outer parts of the annulus, $\calI$ contains all the eigenvalues inside $\calD$. Then the criterion is defined as,
\begin{equation} \label{eq: ratio}
    \frakR = \frac{\sup_{z \in \calO} |R_k(z)|}{\inf_{z \in \calI} |R_k(z)|}.
\end{equation}
Hence, we would like to address the following optimization problem to obtain the optimal weights and poles for a given $k$,
\begin{equation} \label{eq: infratio}
    \inf_{\{w_i^{(k)}\}_{i=1}^{k}, \{p_i^{(k)}\}_{i=1}^{k}}
    \frac{\sup_{z \in \calO} |R_k(z)|}{\inf_{z \in \calI} |R_k(z)|}.
\end{equation}
We call \cref{eq: ratio} as spectrum separation. One could see that as $b/a$ becomes larger, it is easier to separate the values inside and outside the annulus with rational functions. The drawback of using a larger $b$ is that more eigenvalues may fall into the annulus $\{z:a\le|z|\le b\}$ and we do not explicitly know the impact of these eigenvalues on the convergence of the subspace iteration.

\subsection{Zolotarev's third problem}
\label{subsec: Zolotarev}

We introduce Zolotarev's third problem with its related theoretical results~\cite{petrushev_popov_1988, STARKE1992115}. Zolotarev's third problem is about the optimal separation of rational function on two disjoint regions. Since contour discretization yields a rational function, it is natural to bridge the contour discretization and Zolotarev's third problem.

\begin{definition} \label{def: Zolo3rd}
    Let $\calE$ and $\calG$ be two disjoint regions of $\bbC$, i.e., $\calE \cap \calG =\emptyset$. The Zolotarev's third problem is to solve the following optimal problem
    \begin{equation}\label{eq: Zolo3rd}
        Z_k(\calE, \calG) = \inf_{r \in \calR_{k,k}} \frac{\sup_{z \in \calE}
        |r(z)|}{\inf_{z \in \calG} |r(z)|}.
    \end{equation}
\end{definition} 

\begin{figure}[htp]
    \centering
    \includegraphics[scale=0.7]{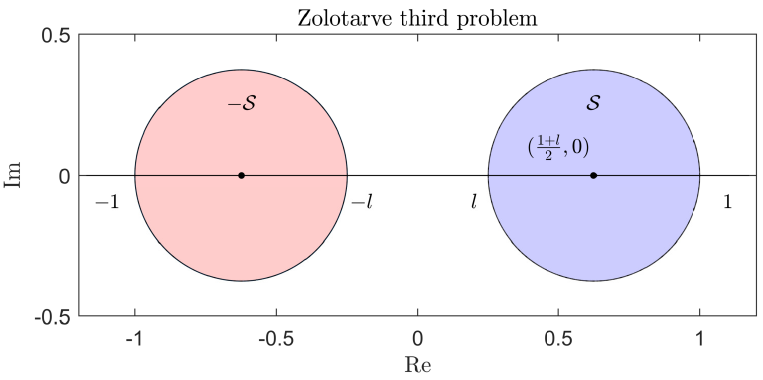}
    \caption{Regions in Zolotarev's third problem when $E$ and $G$ are
    symmetric disks.}
    \label{fig: zolo3rd}
\end{figure}

Zolotarev's third problem tends to find a rational function whose value on $\calE$ is least from zero and the value on $\calG$ is most away from zero. When $\calE$ and $\calG$ are two symmetric disks with respect to the origin as in \cref{fig: zolo3rd}, the solution to Zolotarev's third problem is explicitly given in \cref{thm:explicitzolo3}. \Cref{thm:explicitzolo3} shown in this paper takes a different parameterized form of that in~\cite{STARKE1992115}.

\begin{theorem} \label{thm:explicitzolo3}
    Let $\calS = \{z \in \bbC: |z - \frac{1 + \ell}{2}| \le \frac{1 -
    \ell}{2}\}, 0 < \ell < 1$. Then the rational function
    \begin{equation*}
        R_k^{(Z)}(z) = \bigg(\frac{z - \sqrt{\ell\,}}{z + \sqrt{\ell\,}}
        \bigg)^k,
    \end{equation*} 
    attains the infimum of the Zolotarev's third
    problem $Z_k(\calS,-\calS)$ and the infimum equals to $(\frac{1 +
    \sqrt{\ell\,}}{1 - \sqrt{\ell\,}})^{-2k}$.
\end{theorem}

The explicit solution to Zolotarev's third problem as in \cref{thm:explicitzolo3} is the key to proving the asymptotical optimality in the sense of sepctrm separation with respect to $k$ of the trapezoidal quadrature. The rational function in \cref{thm:explicitzolo3} is referred to as Zolotarev's third function in the rest of this paper.

\subsection{Compact form for $R_{0,1,k}(z)$}
\label{subsec:compactform}

In order to connect the Zolotarev's function and the trapezoidal quadrature of the contour integral, and derive the composite formula in \cref{sec: composite}, we establish an identity that relates $R_{0,1,k}(z)$ and $R_{0,1,1}(z^k)$. The relation heavily relies on the symmetry of the trapezoidal quadrature on the circle.

Let us start with toy cases $k = 2, 4$. The trapezoidal quadrature of the unit circular contour with two poles could be rewritten as
\begin{equation*}
    R_{0,1,2}(z) = \frac{1}{2}\bigg( \frac{e^{\frac{\imath \pi}{2}}}{e^{
    \frac{\imath \pi}{2}} - z} + \frac{e^{\frac{3 \imath \pi}{2}}}{e^{
    \frac{3 \imath \pi}{2}} - z}\bigg)
    = \frac{1}{2} \frac{2 e^{\imath \pi }}{e^{\imath \pi} - z^2}
    = \frac{1}{1 + z^2} = R_{0,1,1}(z^2).
\end{equation*}
Here we use the symmetry of poles and weights with respect to the origin to derive the compact form of $R_{0,1,2}(z)$ and find that $R_{0,1,2}(z)$ is equivalent to $R_{0,1,1}(z^2)$. Let us further derive the compact form of $R_{0,1,4}(z)$,
\begin{equation*}
    \begin{split}
        R_{0,1,4}(z) & = \frac{1}{4} \bigg( \frac{e^{\frac{\imath \pi}{4}}}{e^{
        \frac{\imath \pi}{4}} - z} + \frac{e^{\frac{7\imath \pi}{4}}}{e^{
        \frac{7\imath \pi}{4}} - z} + \frac{e^{\frac{3\imath \pi}{4}}}{e^{
        \frac{3\imath \pi}{4}} - z} + \frac{e^{\frac{5\imath \pi}{4}}}{e^{
        \frac{5\imath \pi}{4}} - z} \bigg)\\
        & = \frac{1}{2} \bigg( \frac{e^{\frac{\imath \pi}{2}}}{e^{
        \frac{\imath \pi}{2}} - z^2} + \frac{e^{\frac{3\imath \pi}{2}}}{e^{
        \frac{3\imath \pi}{2}} - z^2} \bigg)
        = R_{0,1,2}(z^2) = R_{0,1,1}(z^4),
    \end{split}
\end{equation*}
where, in the second equality, we combine the first two and last two terms, and in the last equality, we adopt the compact form of $R_{0,1,2}(z)$. From the derivation of the compact forms of $R_{0,1,2}(z)$ and $R_{0,1,4}(z)$, we could directly extend the derivation to obtain the compact form of $R_{0,1,k}(z) = R_{0,1,1}(z^k)$ for $k = 2^m$, $m\in\bbN_+$. Fortunately, the compact form holds for any $k \in \bbN_+$. The result is summarized in
\cref{lem: compact}.
\begin{lemma} \label{lem: compact}
    For all $k \in \bbN_+$, let $k$ roots of $z^k = -1$ be
    $\sigma^{(k)}_i$ for $i = 1, \dots, k$. Then the compact form of
    $R_{0,1,k}(z)$ admits,
    \begin{equation} \label{eq: trapezoidalquad}
        R_{0,1,k}(z) = \frac{1}{k} \sum_{i = 1}^k \frac{\sigma^{(k)}_i}
        {\sigma^{(k)}_i - z} = \frac{1}{1 + z^k} = R_{0,1,1}(z^k).
    \end{equation} 
\end{lemma}

\begin{proof}

We first prove two equalities, \cref{eq: rootspoly} and \cref{eq: unitequal}, and then derive the compact form of $R_{0,1,k}(z)$.

The $k$ roots of the $k$-th degree polynomial $z^k + 1$ are denoted as $\sigma_i^{(k)}$ for $i = 1, 2, \dots, k$. A $k$-th order polynomial with $k$ roots takes form, $a_k \prod_{i = 1}^k (z - \sigma_i^{(k)})$, where $a_k$ is the coefficient in the leading order. Comparing with the leading order coefficient in $z^k + 1$, we know $a_k = 1$ and have,
\begin{equation} \label{eq: rootspoly}
    z^k + 1 = \prod_{i = 1}^k (z - \sigma_i^{(k)}).
\end{equation}

Then we prove the second equality,
\begin{equation} \label{eq: unitequal}
    -\frac{1}{k} \sum_{i=1}^k \sigma_i^{(k)} \prod_{j = 1, j \not= i}^k (z
    - \sigma_j^{(k)}) = 1.
\end{equation}
The left-hand side of \cref{eq: unitequal} is a $(k - 1)$-th degree polynomial. For equality \cref{eq: unitequal} to hold, we only need to make sure that the equality holds on $k$ different points.Specifically, we examine that on $\sigma_i^{(k)}$ for $i = 1, \dots, k$ and obtain, 
\begin{equation*}
    -\frac{\sigma_i^{(k)}}{k} \prod_{j = 1, j \not= i}^k(\sigma_i^{(k)}
    - \sigma_j^{(k)}) = - \frac{\sigma_i^{(k)}}{k} \lim_{z \to \sigma_i^{(k)}} \frac{z^k
    + 1}{z - \sigma_i^{(k)}} = - \frac{\sigma_i^{(k)}}{k}
    \frac{k (\sigma_i^{(k)})^{k - 1}}{1} = 1,
\end{equation*}
where the first equality is due to \cref{eq: rootspoly} and the continuity of $(z^k + 1) / (z - \sigma_i^{(k)})$, the second equality comes from the L'Hopital rule of complex functions, and the last equality holds since $\sigma_i^{(k)}$ is a root of $z^k + 1$.

Finally, we derive the compact form of $R_k(z)$ as in \cref{lem: compact}.
\begin{equation*}
    \begin{split}
    R_{0,1,k}(z) & = \frac{1}{k} \sum_{i = 1}^k \frac{\sigma_i^{(k)}}{\sigma_i^{(k)} - z}
    = \frac{-\frac{1}{k} \sum_{i = 1}^k \sigma_i^{(k)} \prod_{j = 1, j
    \not= i}^k(z - \sigma_j^{(k)})}{\prod_{i = 1}^k(z - \sigma_i^{(k)})}\\
    & = \frac{1}{\prod_{i = 1}^k (z - \sigma_i^{(k)})} = \frac{1}{z^k + 1}
    = R_{0,1,1}(z^k),
    \end{split}
\end{equation*}
where the second equality adopts \cref{eq: unitequal} and the fourth equality adopts \cref{eq: rootspoly}.
\end{proof}

A related compact form without detailed derivation could be found in~\cite{IKEGAMI20101927}. The compact form \cref{lem: compact} could be further generalized to $R_{c,r,k}(z)$ and results the compact form,
\begin{equation*}
    R_{c,r,k}(z) = \frac{1}{1 + (\frac{z - c}{r})^k}.
\end{equation*}

\subsection{Optimal solution and the asymptotic optimality of trapezoidal quadrature}
\label{subsec: Optimalthm}

In this section, we prove that, if we know the desired spectrum explicitly, the rational function used in the inverse power method achieves the optimal of \cref{eq: Zolo3rd} for $\calE = \calO$ and $\calG = \calI$. On the other hand, the rational function $R_{c,r,k}(z)$ from the trapezoidal quadrature discretization of the contour integral achieves asymptotic optimality of \cref{eq: Zolo3rd}.

\begin{theorem} \label{thm: Zolotarev1}
    The rational function $z^{-k}$ achieves the infimum of~\cref{eq: Zolo3rd} for $\calE=\calO$ and $\calG=\calI$. And the infimum equals to $\big(\frac{a}{b}\big)^k$.
\end{theorem}

\begin{proof}

We address Zolotarev's third problem with region $\calI$ and $\calO$, i.e., $Z_k(\calO,\calI)$. Define a M\"obius transform $M(z) = \gamma \frac{z - \alpha}{z - \beta}$ such that
\begin{equation*}
    M(-b) = 1, \quad M(-a) = -1, \quad M(a) = -\ell, \quad M(b) = \ell.
\end{equation*}
The parameters $\gamma$, $\alpha$, $\beta$, and $\ell$ are determined by $a$ and $b$. They satisfy
\begin{equation*}
    \alpha = \sqrt{ab}, \quad \beta = -\sqrt{ab}, \quad
    \gamma = \frac{\sqrt{b} - \sqrt{a}}{\sqrt{b} + \sqrt{a}}, \quad
    \ell = \bigg( \frac{\sqrt{b} - \sqrt{a}}{\sqrt{b} + \sqrt{a}} \bigg)^2.
\end{equation*}
It can be verified that $M(\calI) = -\calS$ and $M(\calO) = \calS$ for $\calS$ in \cref{thm:explicitzolo3}. Then the composition of the M\"obius transform and  Zolotarev's function $R_k^{(Z)}(M(z))$ achieves the infimum of $Z_k(\calO, \calI)$ and is denoted as, 
\begin{equation}\label{eq: rkaz}
    R_k^{(A)}(z) = R_k^{(Z)}(M(z)) = z^{-k}.
\end{equation}
The infimum of $\calI$ is taken when $|z|=a$ and the supremum of $\calO$ is taken when $|z|=b$. Then the infimum of the ratio is $\big(\frac{a}{b}\big)^k$.
\end{proof}

\Cref{thm: Zolotarev1} gives the optimal rational function in solving \cref{eq: Zolo3rd}. The rational function $z^{-k}$ therein combined with subspace iteration corresponds to the well-known inverse power method. Further, the radius of $\calD$ or the diameter of the annulus is not included in the optimal rational function. Hence, we conclude that, in the sense of convergence rate of subspace iteration, the optimal interior eigensolver is the inverse power method, assuming the center of the desired region $\calD$ is explicitly known.

While the optimal rational function $z^{-k}$ only has one pole and could not be written as a sum of low-order rational functions form~\cref{eq: discretizedCI}. The inverse power method then has to be executed sequentially and could not benefit from the parallelization of distinct poles. In the following, we argue that, although the trapezoidal quadrature of contour integral is not the optimal rational function, it achieves asymptotic optimality. 

We consider that the contour is the boundary of $\calI$. By \cref{lem: compact}, the discretization can be rewritten as 
\begin{equation*}
    R_{0,a,k}(z) = \frac{1}{1 + (\frac{z}{a})^k}.
\end{equation*} 
By the maximum modulus principle, the infimum of $\calI$ and the supremum of $\calO$ are taken when $|z|=a$ and $|z|=b$. In region $\calI$, $|\frac{z}{a}|^k\le 1$. The absolute value of the denominator can be viewed as the distance between $-1$ and $(\frac{z}{a})^k$. By simple computation, the infimum is achieved when $z=a$. Similarly, the supremum of $\calO$ is achieved when $z=\sqrt[k]{-1}b$ from the fact that $|\frac{z}{a}|^k>1$ in $\calO$. The spectrum separation \cref{eq: ratio} is 
\begin{equation*}
    \frakR = \frac{2}{(\frac{b}{a})^k-1} \sim 2 \left( \frac{a}{b} \right)^k,
\end{equation*}
which asymptotically decays with respect to $k$ at the same rate as that in \cref{thm: Zolotarev1}. The above discussion is summarized in the following corollary.

\begin{corollary}
    The trapezoidal quadrature discretization of the contour integral on the boundary of $\calG = \calI$ results in the rational function 
    \begin{equation}
        1\bigg/\bigg[1 + \big(\frac{z}{a}\big)^k\bigg].
    \end{equation} 
    The spectrum separation \cref{eq: ratio} of the trapezoidal quadrature is $2/[(\frac{b}{a})^k-1]$, which achieves the same decay rate as the infimum of~\cref{eq: Zolo3rd} for $\calE=\calO$ and $\calG=\calI$.
\end{corollary}

Although the trapezoidal quadrature used to approximate the contour integral is not the optimal rational function of \cref{eq: Zolo3rd}, the ratio asymptotically achieves the optimal one up to a constant prefactor 2. Hence, we call $R_{c,r,k}(z)$ as the nearly optimal rational function for \cref{eq: Zolo3rd}. The
advantage of the trapezoidal quadrature over $z^{-k}$ is that
$R_{c,r,k}(z)$ could be efficiently parallelized. Specifically, the
solving of the $k$ shifted linear systems $(p_i^{(k)} B - A)x = y$ could be
performed simultaneously, while the solving of $z^{-k}$ has to be
performed subsequently.

\section{Composite rule of trapezoidal quadrature} \label{sec: composite}

In this section, a composite rule of trapezoidal quadrature is derived. Let us start from $c = 0$ and $r = 1$. We repurpose $R_k(z)$ as $R_{0,1,k}(z)$ in the rest of this paper.

Given a positive integer $k$ and its integer factorization $k = k_1 k_2$ for
$k_1 > 1$ and $k_2 > 1$, we aim to rewrite the $k$-th order rational function $R_k(z)$ as a composition of two $k_1$-th and $k_2$-th rational functions, $R_{k_1}(z)$ and $\hat{R}_{k_2}(z) = R_{k_2}(M(z))$, where $M(\cdot)$ is a M\"{o}bius transform function. Precisely, the composite function admits, $R_{k}(z) = \hat{R}_{k_2}(R_{k_1}(z)) = R_{k_2}(M(R_{k_1}(z))$.

According to
\cref{lem: compact}, we have a natural composite expression as,
\begin{equation*}
    R_{k_1k_2}(z) = R_1(z^{k_1 k_2}) = R_{k_2}(z^{k_1}).
\end{equation*}
For the desired composite rule holds, we should let $M(R_{k_1}(z)) = z^{k_1}$. Now we determine the coefficients of $M(z)=(a z - b)/(c z - d)$ such that
$M(R_{k_1}(z)) = z^{k_1}$ holds. For $|z|\not=\infty$, substituting $R_{k_1}(z) = 1/(1 +
z^{k_1})$ into the expression of $M(z)$, we obtain,
\begin{equation*}
    \begin{split}
        & M(R_{k_1}(z)) = \frac{a - b(1 + z^{k_1})}{c - d(1 + z^{k_1})}
        = z^{k_1} \\
        \Longleftrightarrow \quad &
        d z^{2k_1} + (d-c-b) z^{k_1} + (a - b) = 0.
    \end{split}
\end{equation*}
The above equality holds for all $z$. Hence we have solutions of coefficients satisfying $d = 0$ and $a = b = -c$. These solutions of coefficients lead to the unique M{\"o}bius transform function,
\begin{equation} \label{eq: mobiustrans}
    M(z) = \frac{1-z}{z}.
\end{equation}
One can verify that $M(R_{k_1}(z)) = z^{k_1}$ holds for $|z|=\infty$. In \cref{fig: mobius}, the mapping of $R_{k_1}(z)$ and $M(R_{k_1}(z))$ are illustrated.

\begin{figure}[htbp]
    \centering
    \begin{minipage}{0.48\linewidth}
        \includegraphics[scale=0.4]{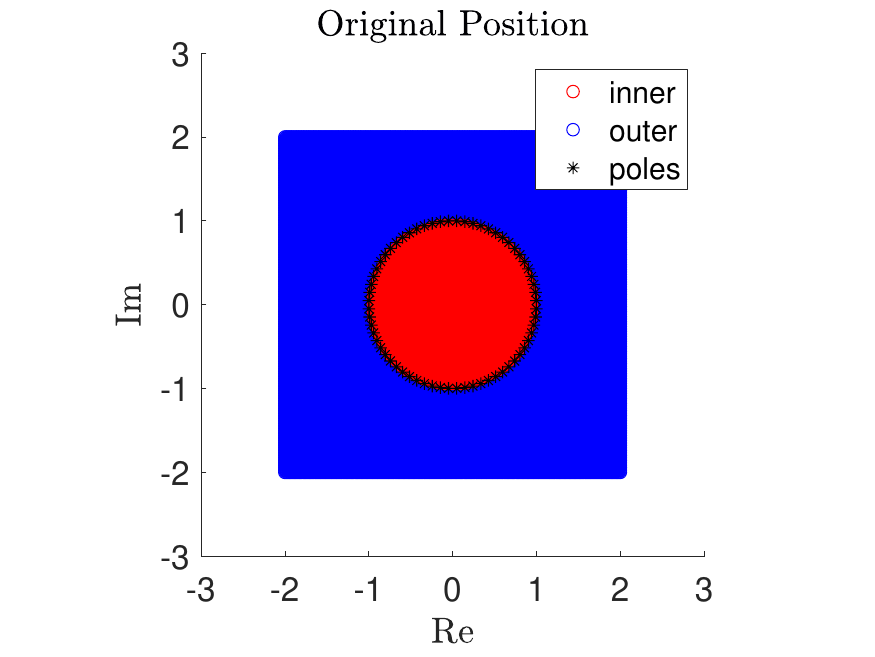}
    \end{minipage}
    \begin{minipage}{0.48\linewidth}
        \includegraphics[scale=0.4]{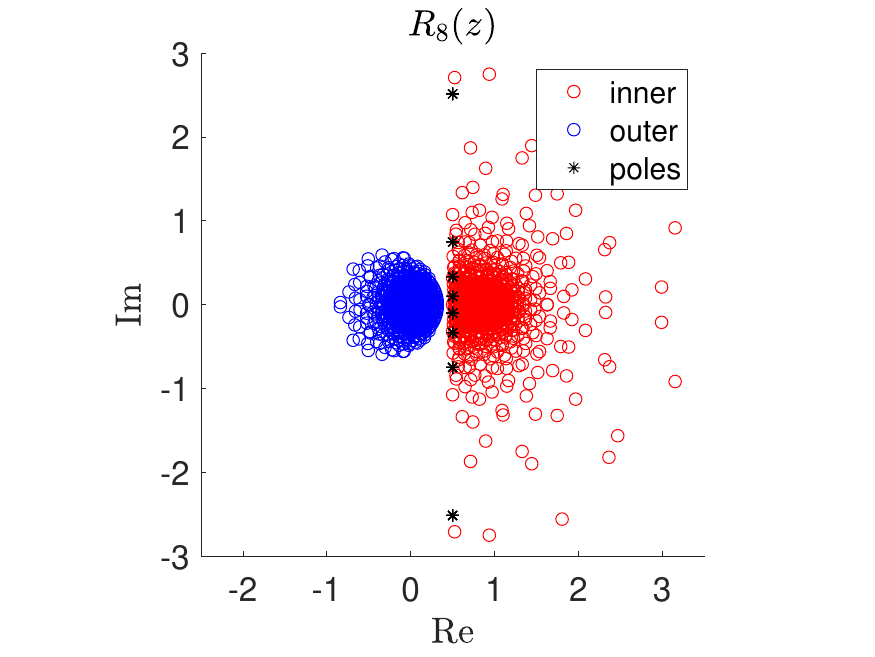}
    \end{minipage}
    \hfill
    \begin{minipage}{0.48\linewidth}
        \includegraphics[scale=0.4]{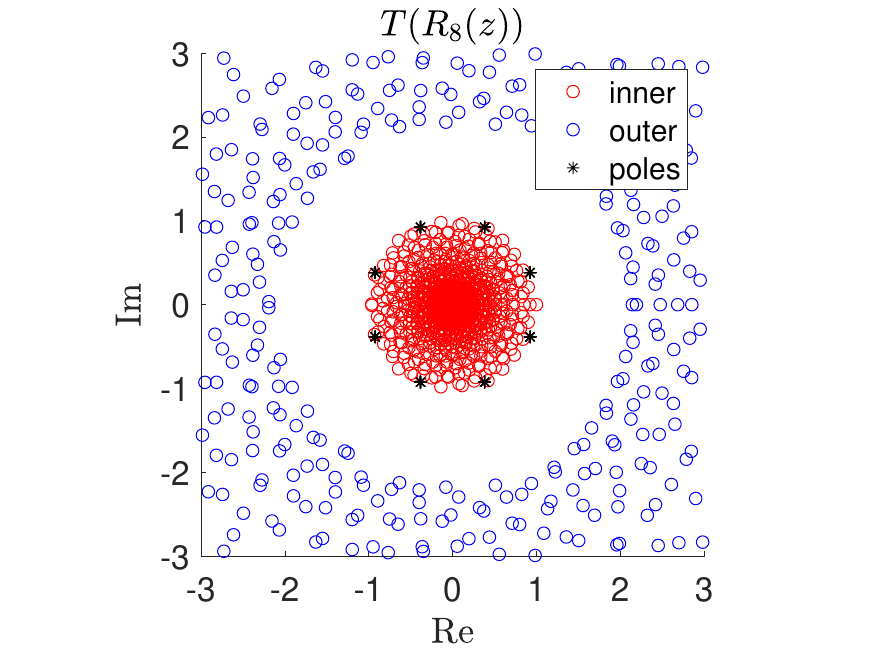}
    \end{minipage}
    \begin{minipage}{0.48\linewidth}
        \includegraphics[scale=0.4]{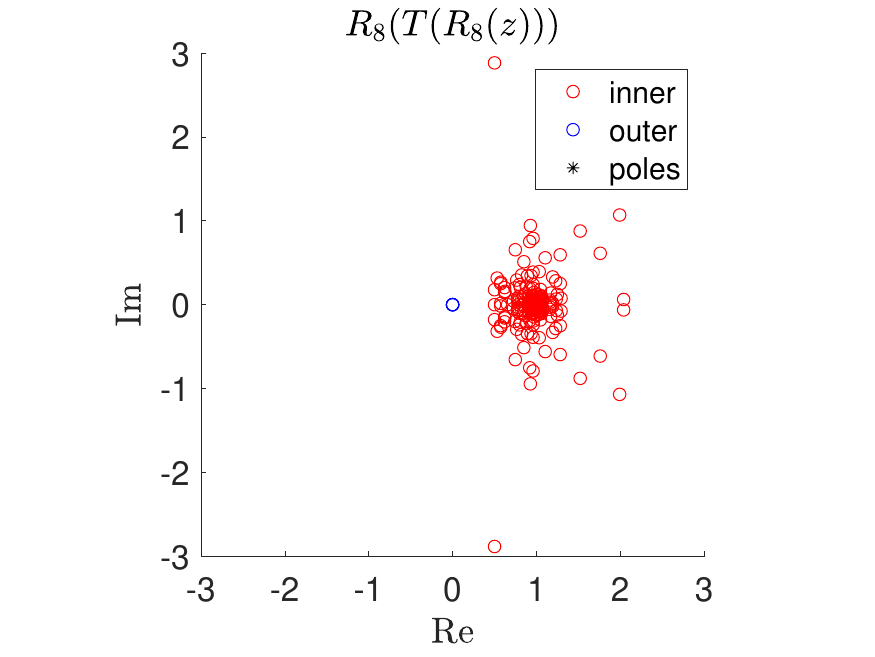}
    \end{minipage}
    \caption{We plot the mapping on $[-2, 2] + [-2, 2] * \imath$. There are $201$ equally spaced points in the direction of the real part and the imaginary part, $40401$ points in total. The outer points are those $|z|>h=1.1$ and the inner points are those $|z|\le 1$ where the contour is $|z|=1$. We fix the figure window at $[-3, 3]+ [-3, 3]*\imath$ except for top right figure which is shown at $[-2.5, 3.5]+ [-3, 3]*\imath$. We let $k_1=k_2=8$ and the poles in all figures are the poles of $R_{k_1k_2}(z)$. The original eigengap is almost invisible, see the top left figure. From the top right figure, $R_8(z)$ maps the inner part to be close to 1 while the outer part to be close to 0 and the poles are mapped to the line $\mathrm{Real}(z)=0.5$. A more clear comparison of pre and post-mapping eigengaps are shown as the difference between the top left figure and bottom left figure. The composite mapping successfully maps the outer part close to 0 and the inner part close to 1 or modulus greater than 1, see bottom right figure.}
    \label{fig: mobius}
\end{figure}

Throughout the above derivation, we conclude that $R_{k_1k_2}(z) = R_{k_2}(M(R_{k_1}(z)))$. A generalized composite rule is given in \cref{thm: Rcrm} for $\Gamma$ with center $c$ and radius $r$. In \cref{thm: Rcrm}, we compose $R_{0,1,k_2}(\cdot)$ and $M(\cdot)$ together and rewrite it as the sum of first-order rational functions. Such a summation form could later be used directly in the algorithm design.
\begin{theorem} \label{thm: Rcrm}
    Given a positive integer $k$ and its integer factorization $k = k_1
    k_2$, the rational function $R_{c,r,k}(z)$ admits the following
    composite rule,
    \begin{equation*}
            R_{c,r,k}(z) = R_{0,1,k_2}(M ( R_{c,r,k_1}(z)) ),
    \end{equation*}
    where $M(\cdot)$ is the M{\"o}bius transform \cref{eq: mobiustrans}.
    When $k_2$ is even, the rational function $R_{c,r,k}(z)$ further
    admits the summation form,
    \begin{equation} \label{eq: Rcrm}
        \begin{aligned}
            &R_{c,r,k}(z) =
            \sum_{i=1}^{k_2} c_i^{(k_2)}\big(R_{c,r,k_1}(z) - s_i^{(k_2)}\big)^{-1}R_{c,r,k_1}(z),\\
            &c_i^{(k_2)} = -\frac{1}{k_2}\frac{\sigma_i^{(k_2)}}{1+\sigma_i^{(k_2)}},~s_i^{(k_2)} = \frac{1}{1+\sigma^{(k_2)}_i},
        \end{aligned}
    \end{equation}
    where $\{\sigma_i^{(k_2)}\}_{i=1}^{k_2}$ are roots of $x^{k_2}=-1$. When $k_2$ is odd, 
    \begin{equation} \label{eq: Rcrmodd}
        R_{c,r,k}(z) =
        \sum_{i=1}^{k_2-1} c_i^{(k_2)}\big(R_{c,r,k_1}(z) - s_i^{(k_2)}\big)^{-1}
        R_{c,r,k_1}(z) + \frac{1}{k_2}R_{c,r,k_1}(z),
    \end{equation}
    where $\sigma^{(k_2)}_{k_2}=-1$.
\end{theorem}

\begin{proof}

    We can use the equation $z = ry + c$ to transfer the contour discretization on an arbitrary circle into the case of the unit circle around the origin. The rational function then admits,
    \begin{equation}
        R_{c,r,k}(z)=R_{0,1,k}(y).
    \end{equation}
    Combining with $R_{0,1,k_1k_2}(z)=R_{0,1,k_2}(M(R_{c,r,k_1}(z)))$, we have
    \begin{equation}
        R_{c,r,k}(z) = R_{0,1,k_1k_2}(y) = R_{0,1,k_2}(M(R_{0,1,k_1}(y)))
        = R_{0,1,k_2}(M(R_{c,r,k_1}(z))).
    \end{equation}
    Now we turn to prove the summation form. When $k_2$ is even,  $\sigma_i^{(k_2)} \neq -1$ holds. With \cref{lem: compact}, the summation form is, 
    \begin{align}
        \begin{split}
            R_{c,r,k_1k_2}(z)  =& R_{0,1,k_2}(M(R_{c,r,k_1}(y)))
            = \frac{1}{k_2} \sum_{i=1}^{k_2}
            \frac{\sigma_i^{(k_2)}}{\sigma_i^{(k_2)} - \frac{1 - 
            R_{c,r,k_1}(y)}{R_{c,r,k_1}(y)}}\\
            = & \frac{1}{k_2} \sum_{i = 1}^{k_2}\frac{\sigma_i^{(k_2)}
            R_{c,r,k_1}(y)}{(1+\sigma_i^{(k_2)})R_{c,r,k_1}(y)-1}\\
            = & \frac{1}{k_2} \sum_{i = 1}^{k_2}
            \frac{\sigma_i^{(k_2)}}{1+\sigma_i^{(k_2)}}
            (R_{c, r, k_1}(z) - \frac{1}{1 + \sigma_i^{(k_2)}})^{-1}
            R_{c, r, k_1}(x)\\
            = & \sum_{i = 1}^{k_2} c_i^{(k_2)}(s_i^{(k_2)} - R_{c, r, k_1}(z))^{-1}
            R_{c, r, k_1}(z), \\
        \end{split}
    \end{align}
    where
    \begin{equation}
        c_i^{(k_2)} = -\frac{1}{k_2} \frac{\sigma_i^{(k_2)}}{1+\sigma_i^{(k_2)}},
        \quad s_i^{(k_2)} = \frac{1}{1+\sigma_i^{(k_2)}}.
    \end{equation}
    When $k_2$ is odd, the term associated with $\sigma_{k_2}^{(k_2)}=-1$ in summation form is equal to $\frac{1}{k_1}R_{c,r,k_1}(z)$.
\end{proof}

The poles of the rational function $R_{c,r,k}(z)$ are transferred into the poles of $R_{0,1,k_2}(M(z))$ by the inner operator $R_{c,r,k_1}(z)$. It is detailed in \cref{prop: poles}.
\begin{proposition} \label{prop: poles}
    For any $p_i^{(k)}$ being a pole of $R_{c,r,k}(z)$, there exist 
    $s_j^{(k_2)}$ for $1\le j\le k_2$, such that
    \begin{equation}
        R_{c,r,k_1}(p_i^{(k)})=s_j^{(k_2)},
    \end{equation}
    where $s_{k_2}^{(k_2)}$ could be infinite when $k_2$ is odd.
\end{proposition}

\begin{proof}
    By \cref{lem: compact}, we know
    \begin{equation}
        R_{c,r,k_1}(p_i^{(k)}) = R_{0,1,k_1}(\sigma_i^{(k)})
        = \frac{1}{1+(\sigma_i^{(k)})^{k_1}} = \frac{1}{1+\sigma_j^{(k_2)}}
        = s_j^{(k_2)}.
    \end{equation} 
\end{proof}

\section{Numerical experiment}\label{sec: experiment}

In this section, we will demonstrate the efficiency and stability of the two algorithms through three experiments. The first experiment shows the advantage of the trapezoidal quadrature over another quadrature, Gauss quadrature. The latter two experiments show the computational benefit of applying \cref{alg: composite} and \cref{alg: composite2}. This paper focuses on filter design rather than proposing a novel projection technique. Hence the projection techniques used in \cref{alg: composite}, \cref{alg: composite2}, and simple rule remain identical. Since the estimation of the number of eigenvalues is beyond the scope of this paper, we assume $s$ is known and set the number of columns $n_{\mathrm{col}}>s$ in all numerical experiments. 

In our experiments, we set the parameters for convergence criterion that have been discussed in \hyperref[subsec: eigensolver]{section~\ref*{subsec: eigensolver}} to $\tau_g = 10^{-2}$ and $\tau = 10^{-8}$. The former one serves as the tolerance for distinguishing the ghost eigenvalues, while the latter one is the target precision of eigenpairs. The $\sigma$ in HSRR is set as $c$, which is the center of the circular contour $\Gamma$.

The direct solver is the \texttt{lu} function in MATLAB with four outputs under the default setting, which leads to a sparse LU factorization for sparse matrices. The triangular solves are performed by  ``$\backslash$'' in Matlab, which can handle multiple right-hand sides simultaneously. All programs are implemented and executed with MATLAB R2022b and are performed on a server with Intel(R) Xeon(R) Gold 6226R CPU at 2.90 GHz and 1 TB memory. In performance experiments, we report the single-thread wall time.

\subsection{Asymptotically optimal rational filter}

Firstly, we show the spectrum separation \cref{eq: ratio} for the trapezoidal quadrature, Gauss quadrature and the optimal one in \cref{thm: Zolotarev1}. The numerical results are illustrated in \cref{fig: TRvsGS}. Here, we set $a = 1$ and $b = 1.1$. The infimum of $\calI$ and supremum of $\calO$ for Gauss quadrature is not known as a closed form, so we use the discretization of
1000 points in both directions of real and imaginary part on $[-1.5, 1.5] + [-1.5, 1.5] \cdot \imath$ to estimate \cref{eq: ratio} of Gauss quadrature. Only even $k$ is adopted as we perform the Gauss quadrature on the upper semicircle and lower semicircle separately. Such a Gauss quadrature discretization that preserves the symmetry will perform better than the one that performs gauss quadrature on the whole circle directly.  

\begin{figure}
    \centering
    \includegraphics[scale=0.4]{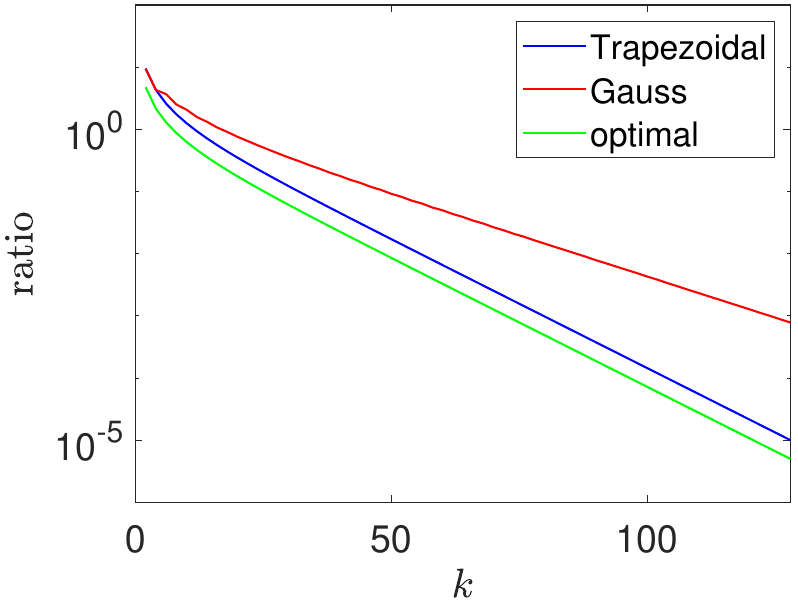}
    \caption{The separation ratio \cref{eq: ratio} for various quadrature rules and numbers of poles. The number of poles $k$ ranges from 2 to 128. The trapezoidal quadrature shows the same slope as the optimal ratio, while the Gauss quadrature behaves differently.}
    \label{fig: TRvsGS}
\end{figure}

From \cref{fig: TRvsGS}, we find that trapezoidal quadrature always outperforms Gauss quadrature. The figure also shows that the trapezoidal quadrature attains the same rate with respect to $k$ as the optimal one, as the slope of the straight line of the trapezoidal quadrature is the same as the optimal one.

We remark that the convergence behavior depends on the distribution of eigenvalues. Our analysis in \cref{sec: Optimal} views the desired spectrum and undesired spectrum as a disk and the complement of a disk. While the eigenvalues of a matrix are discrete points in these regions. There would be the case that the discrete eigenvalues avoid all bad areas in both the numerator and denominator of \cref{eq: ratio} with Gauss quadrature and have a small ratio $\frakR$ of \cref{eq: ratio}. In such a case, the rational filter with Gauss quadrature could outperform the rational filter with trapezoidal quadrature for some matrices. Without prior knowledge of the distribution of eigenvalues, the trapezoidal-quadrature-based filter is a near-optimal choice.

\subsection{Composite rule with subspace iteration}
\label{subsec: comparesimplecomposite}

We compare \cref{alg: composite} against HFEAST with $k_1=k_2=8$ and $k=k_1\cdot k_2=64$.

The class of non-Hermitian generalized eigenvalue problems comes from the model order reduction tasks~\cite{alignment, PRIMA} in the circuit simulation~\cite{circuitsimulation}. Matrices are constructed based on quasi-two-dimensional square power grids of size $n_x \times n_x \times 10$. The non-Hermitian matrix pencil is $(G, C)$, and the pattern and distribution of eigenvalues for $n_x=10$ are shown in \cref{fig: circuit}. One can find the matrix construction details in \cref{sec: construct}. \Cref{tab:circuitinformation} lists information about matrices used in our
numerical experiments as well as their target regions. The last column of \cref{tab:circuitinformation} includes the runtime ratio of one matrix factorization and one triangular solve. In all cases, there are 20 eigenvalues in their target regions and we adopt $n_\mathrm{col} = 24$. Reference eigenvalues are calculated by \texttt{eigs} in MATLAB. The stopping criteria of GMRES is $10^{-9}$. The convergence behaviors are illustrated in \cref{fig: circuiterror}. Runtime is reported in \cref{tab: circuitperformance}. The italic values therein are estimated numbers since the simple rule runs out of memory for those settings. 

\begin{figure}[htb]
    \centering
    \subfigure[]{    
        \begin{minipage}{0.5\linewidth}
        \includegraphics[scale=0.5]{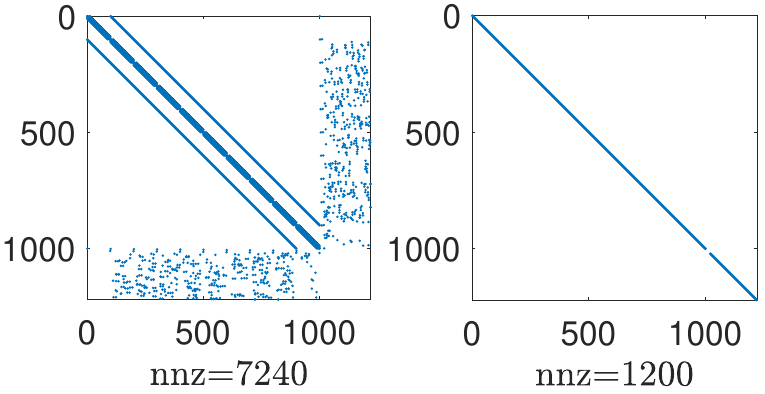}
    \end{minipage}
    }
    \subfigure[]{    
        \begin{minipage}{0.45\linewidth}
        \includegraphics[scale=0.35]{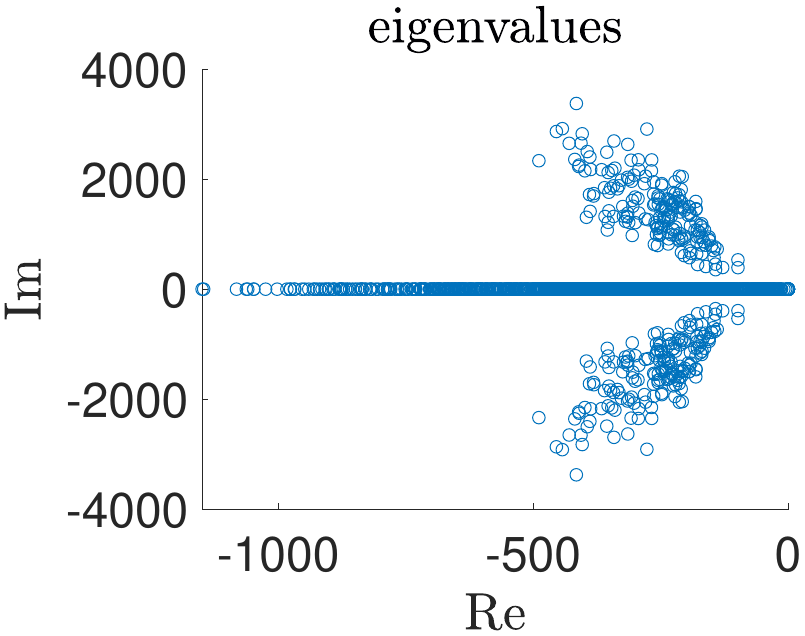}
    \end{minipage}
    } \caption{(a) Patterns of $G$ and $C$ when $n_x = 10$; (b)
    Eigenvalues distribution.} 
    \label{fig: circuit}
\end{figure}

\begin{table}[htb]
    \centering
    \caption{Matrix information. Columns show sizes and the number of nonzeros (nnz) of the $G+C$ matrix for various $n_x$. The centers and radii of target regions are included and each encloses 20 eigenvalues. The last column includes the runtime ratio of one matrix factorization and one triangular solve.}
    \label{tab:circuitinformation}
    \begin{tabular}{crrcr}
        \toprule
        $n_x$ & Size & nnz & $(c,r)$ & $C_{\mathrm{factor}}/C_{\mathrm{apply}}$ \\
        \toprule
        10  & 1,220     & 7,440      & ($-200+1000\imath$,90) & 33.706\\
        100 & 120,020   & 776,040    & ($-101+22\imath$,3)     & 47.903\\
        200 & 480,020   & 3,112,040  & ($-24+4.7\imath$,2.1)    & 71.119\\
        400 & 1,920,020 & 12,464,040 & ($-5.3+1\imath$,0.9)      & 118.037\\
        \bottomrule
    \end{tabular}
\end{table}

\begin{figure}[htb]
    \centering
    \includegraphics[width=0.3\textwidth]{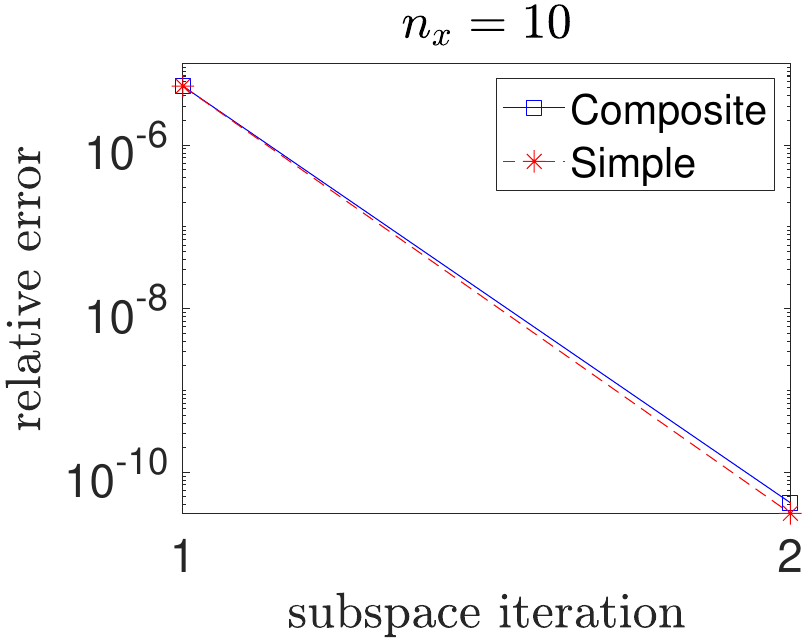}
    \includegraphics[width=0.3\textwidth]{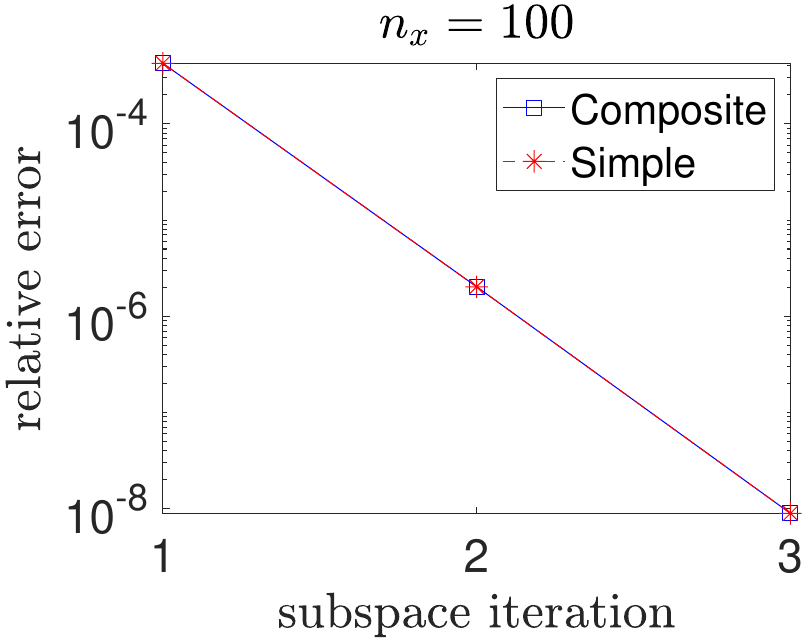}
    \includegraphics[width=0.3\textwidth]{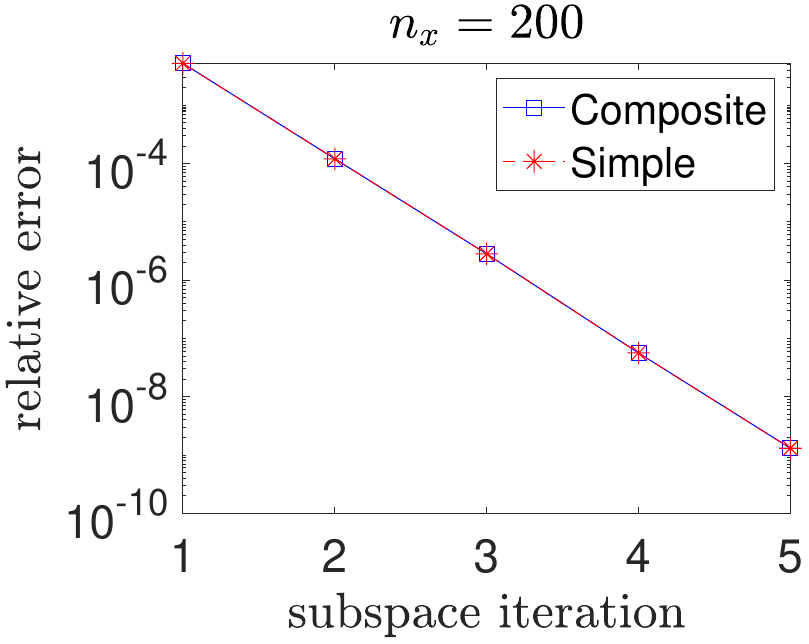}
    \caption{Convergence of the simple rule and the composite rule.} 
    \label{fig: circuiterror}
\end{figure}

\begin{table}
    \centering
    \caption{Runtime (second) of the simple rule and the composite rule for matrices in \cref{tab:circuitinformation}. Italic values are estimated due to the out-of-memory limit. \texttt{Comp} means the composite rule.}
    \label{tab: circuitperformance}
    \begin{tabular}{ccccccc}
        \toprule
        \multirow{2}{*}{$n_x$} & \multicolumn{2}{c}{total}
        & \multicolumn{2}{c}{factorization} & \multicolumn{2}{c}{solving}\\
        \cmidrule(r){2-3}
        \cmidrule(r){4-5}
        \cmidrule(r){6-7}
        & Simple & Comp & Simple & Comp & Simple & Comp \\
        \midrule
        10  &$1.5\times 10^{0}$&$2.8\times 10^{0}$&$6.2\times 10^{-1}$&$7\times 10^{-2}$&$5.6\times 10^{-1}$&$1.9\times 10^{0}$\\
        100 &$1.0\times 10^{3}$&$2.5\times 10^{3}$&$4.1\times 10^{2}$&$5.1\times 10^{1}$&$6.0\times 10^{2}$&$2.3\times 10^{3}$\\
        200 &$9.0\times 10^{3}$&$2.8\times 10^{4}$&$3.4\times 10^3$&$4.2\times 10^2$&$5.6\times 10^3$&$2.6\times 10^4$\\
        400 &$\mathit{5.9\times 10^{4}}$&$2.2\times 10^{5}$&$\mathit{2.9\times 10^{4}}$&$3.6\times 10^{3}$&$\mathit{3.1\times 10^{4}}$&$2.1\times 10^{5}$\\
        \bottomrule
    \end{tabular}
\end{table} 

\Cref{fig: circuiterror} shows that the composite rule converges in a similar fashion to the simple rule. This indicates that both the GMRES and the direct solver achieve sufficiently good accuracy. In most cases we have tested, the subspace iteration converges effectively when many poles are used, i.e., usually in a few iterations.

The composite rule establishes a trade-off between the number of matrix factorizations and the number of triangular solves in GMRES. \Cref{tab: circuitperformance} shows a comparison of the simple rule and the composite rule in two parts: runtime and memory. As shown in the last column of \cref{tab:circuitinformation}, the runtime ratio between the factorization and the triangular solve grows as the matrix size increases, which is because the matrix factorization is of higher order complexity compared to that of the triangular solve. Hence, reducing the number of factorizations as in the composite rule would be beneficial for large matrices. 

However, as shown in all cases of \cref{tab: circuitperformance}, the
simple rule outperforms the composite rule in total runtime, because
the solving time dominates. The domination comes from the increased
number of subspace iterations, which is due to the denser spectrum of
the larger case. It is hard to tell when the composite rule
outperforms the simple rule for a specific case. The guidance is that
when the factorization time dominates, the composite rule can help us
substitute the solving time for the factorization time, which reduces
the total runtime. 

Regarding the memory cost, the simple rule costs
about $k_2$ times more than that of the composite rule. In these
examples, we find that the simple rule with $n_x = 400$ already
exceeds our memory limit, whereas the composite rule could solve
eigenvalue problems with $n_x = 400$ or even larger. Another benefit
of the composite rule is that it allows us to utilize $R_{c,r,k}(B^{-1}A)$ for 
large $k$ with limited memory.

\subsection{Composite rule without subspace iteration}
\label{subsec:expalgorithm2}

This experiment aims to show that with large $k_2$, the composite rule will converge without subspace iteration, and the GMRES iteration number does not increase dramatically when $k_2$ increases. Such an observation means the strategy doubling $k_2$ each time in \cref{alg: composite2} would be affordable compared to the case with optimal $k_2$. Throughout this section, we reuse matrix pencils in \hyperref[subsec: comparesimplecomposite]{section~\ref*{subsec: comparesimplecomposite}}. We perform three algorithms in this section: the simple rule with $k = 8$,
the composite rule with $k_1 = 8$, and various choices of fixed $k_2$ (\cref{alg: composite}), and \cref{alg: composite2} with $k_1 = 8$. Also, various choices of the numbers of columns are explored.

\begin{table}[htb]
    \centering
    \caption{Details of the simple rule and \cref{alg: composite2}. The column $p$ shows the number of filtered eigenpairs, i.e., the number of approximate eigenpairs inside the region whose relative error is less than $\tau_g$. The column $e$ shows the relative error when the algorithm converges or the limitation of subspace iteration is attained. The column $n_\mathrm{iter}$ shows the times of applying $G$ to a set of vectors for the simple rule. While for \cref{alg: composite2}, the column $n_\mathrm{iter}$ shows the maximum GMRES step of different vectors since the GMRES step will change with the vector. When not all eigenpairs are filtered, we use ``-'' for $e$ and $n_\mathrm{iter}$, since the algorithm will fail to filter the eigenpairs of interest even if we run the algorithm with infinite $n_{\mathrm{iter}}$, or $n_{\mathrm{iter}}$ would be no less than 400 for the target precision $\tau$, which can be derived by two equations $e^{\lfloor n_\mathrm{iter}/100\rfloor}=\tau$ and $\sqrt[4]{\tau}=\tau_g<e$.}
    \label{tab: simplevsoneshot}
    \begin{tabular}{ccccccc}
        \toprule
        \multirow{2}{*}{$(n_x,n_\mathrm{col})$} & \multicolumn{3}{c}{Simple}& \multicolumn{3}{c}{Composite}\\
        & $p$ & $e$ & $n_\mathrm{iter}$ & $p$ & $e$ & $n_\mathrm{iter}$ \\
        \midrule
        (100,21) & 19 & -                & -              & 20 &$1.0\times 10^{-10}$& 39\\
        (100,22) & 20 &$7.7\times10^{-9}$& 64             & 20 &$9.6\times 10^{-11}$& 39\\
        (100,24) & 20 &$7.0\times10^{-9}$& 35             & 20 &$4.8\times 10^{-11}$& 39\\
        \midrule
        (200,21) & 20 &$2.0\times10^{-4}$& $\mathit{217}$ & 20 &$3.1\times 10^{-9}$& 51\\
        (200,22) & 20 &$4.2\times10^{-7}$& $\mathit{126}$ & 20 &$1.9\times 10^{-10}$& 51\\
        (200,24) & 20 &$8.2\times10^{-9}$& 57             & 20 &$2.9\times 10^{-11}$& 51\\
        \midrule
        (400,21) & 19 & -                & -              & 20 &$2.3\times 10^{-9}$& 87\\
        (400,22) & 19 & -                & -              & 20 &$3.5\times 10^{-9}$& 87\\
        (400,24) & 20 &$7.0\times10^{-9}$& 46             & 20 &$2.3\times 10^{-9}$& 87\\
        \bottomrule
    \end{tabular}
\end{table} 

\begin{figure}[htb]
    \centering
    \includegraphics[width=0.32\textwidth]{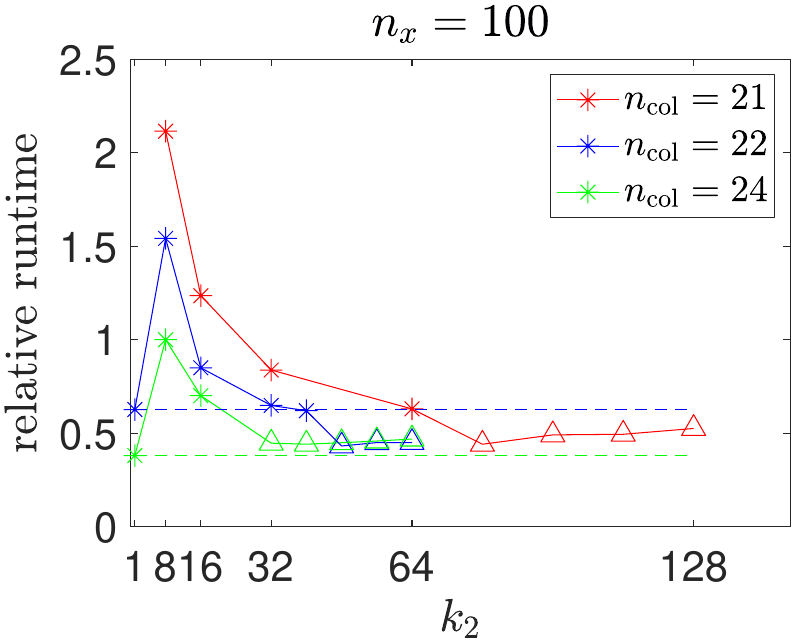}
    \includegraphics[width=0.32\textwidth]{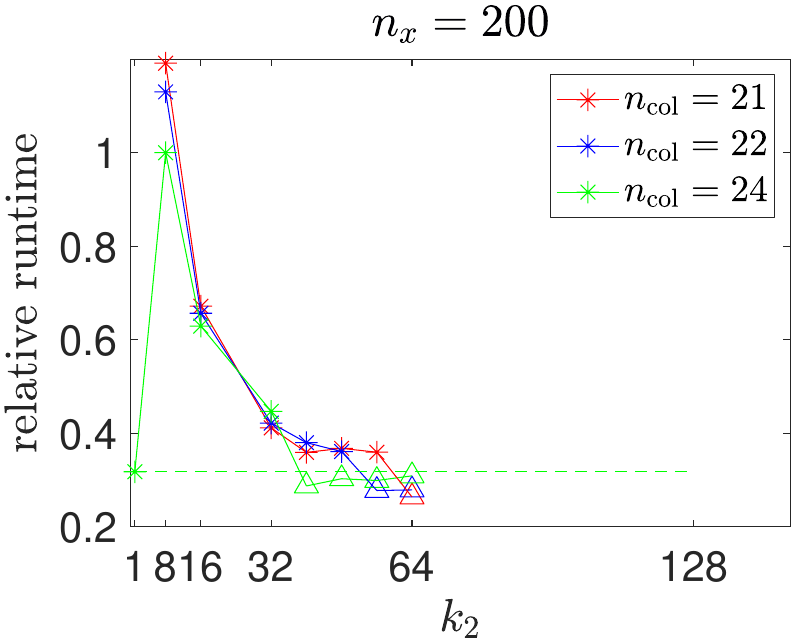}
    \includegraphics[width=0.32\textwidth]{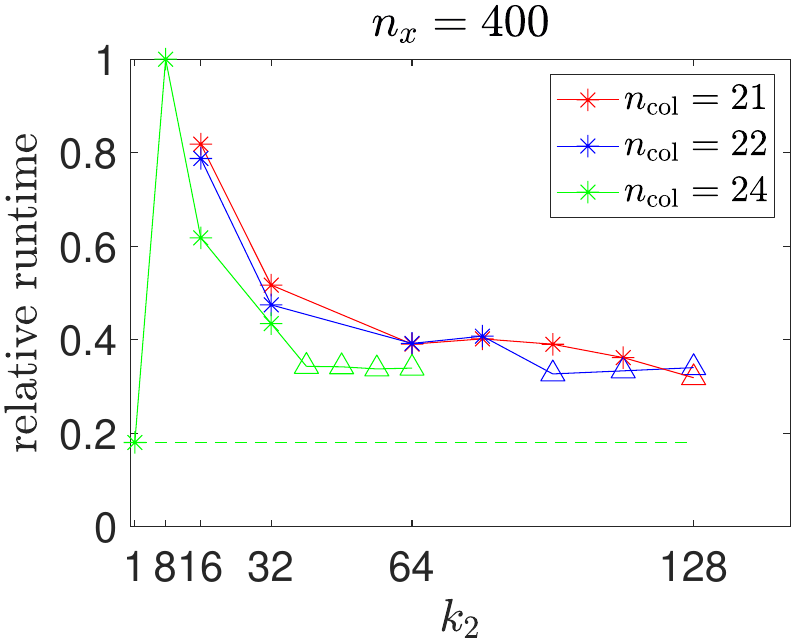}
    \caption{Relative runtime of \cref{alg: composite} with $k_1 = 8$ and various $k_2$. $k_2=1$ represents the simple rule. The runtime is scaled by the runtime of \cref{alg: composite} with $k_1=k_2=8$ and  $n_\mathrm{col}=24$. We do not plot the point that fails to converge.  The star marks denote those subspace iterations converge in more than one iteration, whereas the triangle marks denote those without subspace iteration.}
    \label{fig: subspacevscomposite}
\end{figure}

We set the simple rule to have no more than 100 subspace iterations, while for the composite rule, the limitation is 10. We also terminate \cref{alg: composite} when in the first subspace iteration all eigenpairs converge to target tolerance $\tau$ for relative error. All 20 eigenpairs inside are filtered and the relative errors are less than $\tau$ for the composite rule. \cref{fig: subspacevscomposite} illustrates the relative runtime of \cref{alg: composite} for various $k_2$, and \Cref{tab: simplevsoneshot} reports details of the simple rule and \cref{alg: composite2}. The relative runtime of \cref{alg: composite2} could be read from \cref{fig: subspacevscomposite} from those first triangle marks at $k_2$ being a power of two. \cref{tab: simplevsoneshot} shows more details of the simple rule and \cref{alg: composite2}. We can estimate the $n_\mathrm{iter}$ for the cases of the simple rule that all eigenpairs of interest are filtered but the relative error can not decrease to target $\tau$, from the equation $e^{\lfloor n_\mathrm{iter}/100\rfloor}=\tau$. We use italic numbers to distinguish estimated $n_\mathrm{iter}$ from the
real one.

In \cref{tab: simplevsoneshot}, all three choices of $n_\mathrm{col}$ overestimate the actual number of eigenvalues in the region. The simple rule with a fixed $k = 8$ fails to converge when $n_\mathrm{col}$ is not sufficiently large, e.g., $n_\mathrm{col}=21,22$. In contrast, \cref{alg: composite2} converges in all scenarios. Based on this experiment and other experiments we tried but not listed in the current paper, the convergence of the simple rule is sensitive to the choice of two hyperparameters, $k$ and $n_\mathrm{col}$. While the convergence of \cref{alg: composite2} is more robust. In the worst-case scenario, when the given region is enclosed by many unwanted eigenvalues, extremely large $n_\mathrm{col}$ would be needed to resolve the convergence issue in the simple rule. When the simple rule and \cref{alg: composite2} converge, the latter one outperforms the former one for small $n_\mathrm{col}$, see all the red curves and blue curves in \cref{fig: subspacevscomposite}. From green curves, we know that when $n_\mathrm{col}$ increases, these two methods become comparable on runtime.

\Cref{fig: subspacevscomposite} also explores the optimal choice of $k_2$ without subspace iteration, i.e., the first triangle marks on each curve. We find that the optimal $k_2$ is not necessary $2^p k_1$ as in \cref{alg: composite2}. Besides the factorization cost, the dominant computational cost of the composite rule is the multi-shift GMRES iteration number, i.e., the number of applying $G$~\cref{eq: linearsolve1}. Increasing $k_2$ would add more shifts to the multi-shift GMRES but not necessarily increase iteration number, and the extra cost of
orthogonalization is negligible compared to that of applying $G$. In all curves in \cref{fig: subspacevscomposite}, we observe that, after their first triangle marks, the relative runtime mostly stays flat and increases extremely slowly. Hence, even if \cref{alg: composite2} is not using the optimal $k_2$, the runtime of \cref{alg: composite2} is almost the same as
that with optimal $k_2$. We conclude that \cref{alg: composite2} is an efficient and robust eigensolver and is more preferred than \cref{alg: composite}.

\begin{remark}
    We remark on the hyperparameter choices in \cref{alg: composite2}. Given a matrix pencil and a region, an overestimation $n_\mathrm{col}$ of the number of eigenvalues $s$ is required. If we perform factorizations and triangular solves sequentially, we may need to choose a proper $k_1$ depending on whether factorizations are more expansive than that of triangular solves. In the view of parallelization, the $k_1$ factorizations and the corresponding triangular solves are ideally parallelizable. Hence, we would set $k_1$ as large as possible to fully use the computation resource and reduce the GMRES iterations.
\end{remark}

\section{Conclusion}
\label{sec: conclusion}

This paper finds the optimal rational function in the sense of spectrum separation via Zolotarev's third function. The optimal rational function leads to the traditional inverse power method in numerical linear algebra. Discretizing the contour integral with the standard trapezoidal quadrature results in an asymptotically optimal rational function. Further, we derive the composite rule of the trapezoidal quadrature, i.e., $R_{c,r,k}(z) = R_{0,1,k_2}(M(R_{c,r,k_1}(z)))$ for $k=k_1k_2$ being a positive integer factorization and $M$ being a M\"{o}bius transform.

Based on the composite rule, we propose two eigensolvers for the generalized non-Hermitian eigenvalue problems, \cref{alg: composite} and \cref{alg: composite2}. Both algorithms adopt direct matrix factorizations for the inner rational function evaluation and multi-shift GMRES for the outer rational function. Compared to the simple rule with the same number of poles, both composite-rule-based algorithms reduce the number of factorizations and reduce the memory requirement. This is of fundamental importance when matrices are of large scale. The difference between the
two composite algorithms is the subspace iteration. In \cref{alg: composite}, both $k_1$ and $k_2$ are hyperparameters, and the algorithm adopts the subspace iteration to converge to desired eigenpairs. In contrast, \cref{alg: composite2} is designed without subspace iteration. \cref{alg: composite2} adopts $k_1$ as a hyperparameter and gradually increases $k_2$ until the rational function approximation is accurate enough and the algorithm converges to desired eigenpairs without subspace iteration. As $k_2$ increases in \cref{alg: composite2}, by the property of
multi-shift GMRES, the number of GMRES iterations, i.e., the number of applying $G$, increases very mildly. Hence, compared to the simple rule and \cref{alg: composite}, \cref{alg: composite2} is a robust and efficient eigensolver.

We demonstrate the efficiency of the proposed algorithms by synthetic and practical generalized non-Hermitian eigenvalue problems. Numerical results show that \cref{alg: composite} outperforms the simple rule only if the matrix factorization is much more expensive than the triangular solve. The convergence of \Cref{alg: composite2} is not sensitive to hyperparameter $n_\mathrm{col}$. In terms of the runtime, \cref{alg: composite2} either outperforms or is comparable to the simple rule. A suggestion for the hyperparameter choices of \cref{alg: composite2} is also provided based on both the analysis and numerical results.

\section*{Acknowledgement}

We thank Chao Yang for the helpful discussions. This work is supported in part by the National Natural Science Foundation of China (12271109) and Shanghai Pilot Program for Basic Research - Fudan University 21TQ1400100 (22TQ017).

\appendix

\section{GMRES iteration number}
\label{sec: decayGMRES}

As we mentioned in \cref{subsec: eigensolver}, the multi-shift GMRES will converge faster as the subspace iteration goes. \Cref{tab: circuitbehavior} reports the number of triangular solves in both the simple and the composite rules and its GMRES iteration number. The normalized last column of \cref{tab: circuitbehavior} is visualized in \cref{fig: decays}.

\begin{table}
    \centering
    \small
    \caption{Number of triangular solves of
    the simple rule and composite rule in each subspace iteration are
    shown at the columns of ``Solving'' respectively. The column
    ``$n_\mathrm{iter}$'' shows the number of GMRES iterations in each
    subspace iteration. Italic values are estimated since the required
    memory is beyond our machine memory. }
    \label{tab: circuitbehavior}
    \begin{tabular}{clll}
        \toprule
        \multirow{2}{*}{$n_x$} & Simple &
        \multicolumn{2}{c}{Composite} \\
        \cmidrule(r){3-4}
        & Solving & $n_\mathrm{iter}$ & Solving \\
        \midrule
        10  & [1536,1536]  & [32,22] & [6128,3200]\\
        100 & [1536,1536,1536]& [39,32,31]& [7488,5504,4472]\\
        200 & [1536,1536,1536,1536,1536]& [51,43,38,37,37]& [9680,7784,6504,6296,5392]\\
        400 & $\mathit{[1536,1536,1536,1536,1536]}$ & [86,84,75,67,58]&[16328,13968,10552,6896,3952]\\
        \bottomrule
    \end{tabular}
\end{table} 

\begin{figure}
    \centering
    \includegraphics[scale=0.5]{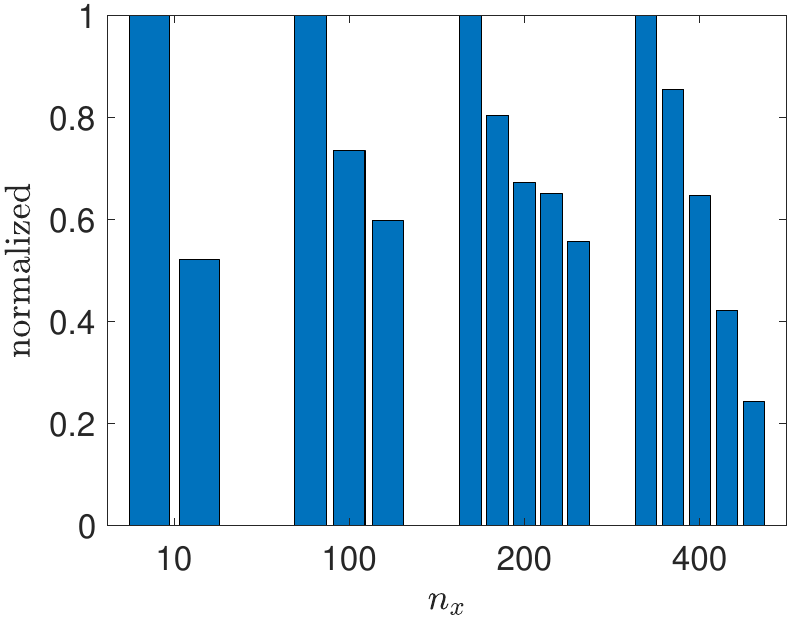}
    \caption{Normalized solving cost in the composite rule. In each test matrix, the bars show the number of triangular solves in each subspace iteration, which are normalized by the number of triangular solves in the first subspace iteration.}
    \label{fig: decays}
\end{figure}

\Cref{tab: circuitbehavior} shows that the number of triangular solves in each subspace iteration of the simple rule stays constant, whereas that for the composite rule decreases. Notice that the $n_\mathrm{iter}$ decays much slower than the number of triangular solves in the composite rule. That is because different columns converge to eigenvectors with different rates.

\section{Construction of matrix}
\label{sec: construct}

Matrices are constructed based on quasi-two-dimensional square power grids of size $n_x \times n_x \times 10$. The non-Hermitian matrix pencil is $(G, C)$ taking the block form as,
\begin{equation*}
    G=\begin{bmatrix}
        G_{11} & G_{12} \\
        G_{21} & 0
    \end{bmatrix}, \quad
    C = \begin{bmatrix}
        C_c & 0 \\
        0 & L
    \end{bmatrix}.
\end{equation*}
In particular, $G_{11}$ represents the conductance matrix as $G_{11} = L_{n_x} \otimes I_{n_x} \otimes I_{10} + I_{n_x} \otimes L_{n_x} \otimes I_{10} + \frac{1}{10} I_{n_x} \otimes I_{n_x} \otimes L_{10}$, where $L_n$ is a weighted one-dimensional Laplacian matrix of size $n \times n$ as 
\begin{equation*}
    L_n = \frac{n}{100}
    \begin{bmatrix}
        1 & -1 & & & \\
        -1 & 2 & -1 & &  \\
        & \ddots & \ddots & \ddots & \\
        &       &   -1  &   2  &  -1 \\
        &       &       &  -1 &   1
    \end{bmatrix}_{n \times n}
\end{equation*}
and $I_n$ is an identity matrix of size $n \times n$. The off-diagonal blocks of $G$ admit $G_{12} = -G_{21}^\top \in \bbR^{10 n_x^2 \times (20 + 2n_x^2)}$ with entries being $\pm 1$ or zero. The first 20 columns of $G_{12}$ correspond to 20 input ports at $(\cdot, 1, 1)$ and $(\cdot, n_x, 10)$ two edges, where the corresponding rows have a positive one. The rest $2n_x^2$ columns of $G_{12}$ correspond to inductors. We uniformly randomly pick $2n_x^2$ interior nodes from grid nodes and add an inductor with their neighbor nodes on the same layer. The corresponding $G_{12}$ part is the incidence matrix of the inductor graph. Matrix $L$ is a diagonal matrix of size $20 + 2n_x^2$. The first $20 \times 20$ block of $L$ is zero. The later $2n_x^2 \times 2n_x^2$ block has diagonal entries uniformly randomly sampled from $[0.5,1.5] \cdot n_x \cdot 10^{-4}$ being the inductance of inductors. The submatrix $C_c$ represents capacitors in the circuit. For each node, we add a grounded capacitor with capacitance uniformly randomly sampled from $[0.5,1.5] \cdot 10^{-3}$, which means $C_c$ is a diagonal matrix whose elements are equal to the capacitances.
\bibliographystyle{siamplain}
\bibliography{references}

\begin{thebibliography}{10}

\bibitem{multihiftGMRES}
{\sc T.~Bakhos, P.~K. Kitanidis, S.~Ladenheim, A.~K. Saibaba, and D.~B. Szyld}, {\em Multipreconditioned gmres for shifted systems}, SIAM Journal on Scientific Computing, 39 (2017), pp.~S222--S247, \url{https://doi.org/10.1137/16M1068694}.

\bibitem{alignment}
{\sc P.~Gross, R.~Arunachalam, K.~Rajagopal, and L.~Pileggi}, {\em Determination of worst-case aggressor alignment for delay calculation}, in 1998 IEEE/ACM International Conference on Computer-Aided Design. Digest of Technical Papers (IEEE Cat. No.98CB36287), 1998, pp.~212--219, \url{https://doi.org/10.1145/288548.288616}.

\bibitem{RIM}
{\sc R.~Huang, J.~Sun, J.~Sun, and C.~Yang}, {\em Recursive integral method with cayley transformation}, Numerical Linear Algebra with Applications, 25 (2017).

\bibitem{CIRR}
{\sc T.~Ikegami and T.~Sakurai}, {\em {Contour integral eigensolver for non-Hermitian systems: a rayleigh-ritz-type approach}}, Taiwanese Journal of Mathematics, 14 (2010), pp.~825 -- 837, \url{https://doi.org/10.11650/twjm/1500405869}.

\bibitem{IKEGAMI20101927}
{\sc T.~Ikegami, T.~Sakurai, and U.~Nagashima}, {\em A filter diagonalization for generalized eigenvalue problems based on the sakurai–sugiura projection method}, Journal of Computational and Applied Mathematics, 233 (2010), pp.~1927--1936, \url{https://doi.org/10.1016/j.cam.2009.09.029}.

\bibitem{dualFEAST}
{\sc J.~Kestyn, E.~Polizzi, and P.~T. Peter~Tang}, {\em Feast eigensolver for non-hermitian problems}, SIAM Journal on Scientific Computing, 38 (2016), pp.~S772--S799, \url{https://doi.org/10.1137/15M1026572}.

\bibitem{Lehoucq1998ARPACKUG}
{\sc R.~B. Lehoucq, D.~C. Sorensen, and C.~Yang}, {\em Arpack users' guide - solution of large-scale eigenvalue problems with implicitly restarted arnoldi methods}, in Software, environments, tools, 1998.

\bibitem{li2021interior}
{\sc Y.~Li and H.~Yang}, {\em Interior eigensolver for sparse hermitian definite matrices based on zolotarev’s functions}, Communications in Mathematical Sciences, 19 (2021), pp.~1113--1135.

\bibitem{QZ}
{\sc C.~B. Moler and G.~W. Stewart}, {\em An algorithm for generalized matrix eigenvalue problems}, SIAM Journal on Numerical Analysis, 10 (1973), pp.~241--256, \url{https://doi.org/10.1137/0710024}.

\bibitem{circuitsimulation}
{\sc F.~N. Najm}, {\em Circuit Simulation}, Wiley-IEEE Press, 2010.

\bibitem{PRIMA}
{\sc Odabasioglu, Celik, and Pileggi}, {\em Prima: passive reduced-order interconnect macromodeling algorithm}, in 1997 Proceedings of IEEE International Conference on Computer Aided Design (ICCAD), 1997, pp.~58--65, \url{https://doi.org/10.1109/ICCAD.1997.643366}.

\bibitem{petrushev_popov_1988}
{\sc P.~P. Petrushev and V.~A. Popov}, {\em Rational Approximation of Real Functions}, Encyclopedia of Mathematics and its Applications, Cambridge University Press, 1988, \url{https://doi.org/10.1017/CBO9781107340756}.

\bibitem{FEASTorigin}
{\sc E.~Polizzi}, {\em Density-matrix-based algorithm for solving eigenvalue problems}, Phys. Rev. B, 79 (2009), p.~115112, \url{https://doi.org/10.1103/PhysRevB.79.115112}.

\bibitem{GMRES}
{\sc Y.~Saad}, {\em Iterative Methods for Sparse Linear Systems}, Society for Industrial and Applied Mathematics, second~ed., 2003, \url{https://doi.org/10.1137/1.9780898718003}.

\bibitem{SAKURAI}
{\sc T.~Sakurai and H.~Sugiura}, {\em A projection method for generalized eigenvalue problems using numerical integration}, Journal of Computational and Applied Mathematics, 159 (2003), pp.~119--128, \url{https://doi.org/10.1016/S0377-0427(03)00565-X}.

\bibitem{STARKE1992115}
{\sc G.~Starke}, {\em Near-circularity for the rational zolotarev problem in the complex plane}, Journal of Approximation Theory, 70 (1992), pp.~115--130, \url{https://doi.org/10.1016/0021-9045(92)90059-W}.

\bibitem{krylovschur}
{\sc G.~W. Stewart}, {\em A krylov--schur algorithm for large eigenproblems}, SIAM Journal on Matrix Analysis and Applications, 23 (2002), pp.~601--614, \url{https://doi.org/10.1137/S0895479800371529}.

\bibitem{HFEAST}
{\sc G.~Yin}, {\em A harmonic feast algorithm for non-hermitian generalized eigenvalue problems}, Linear Algebra and its Applications, 578 (2019), pp.~75--94, \url{https://doi.org/10.1016/j.laa.2019.04.036}.

\bibitem{BFEAST}
{\sc G.~Yin, R.~H. Chan, and M.-C. Yeung}, {\em A feast algorithm with oblique projection for generalized eigenvalue problems}, Numerical Linear Algebra with Applications, 24 (2017), p.~e2092, \url{https://doi.org/10.1002/nla.2092}.

\end{thebibliography}
\end{document}